\documentclass[11pt,reqno]{amsproc}

\title[Lagrangian analyticity]{Analyticity of Lagrangian trajectories for well posed inviscid incompressible fluid models}

\author{Peter Constantin}
\author{Vlad Vicol}
\author{Jiahong Wu}

\address{Department of Mathematics, Princeton University, Princeton, NJ 08544}
\email{const@math.princeton.edu}

\address{Department of Mathematics, Princeton University, Princeton, NJ 08544}
\email{vvicol@math.princeton.edu}

\address{Department of Mathematics, Oklahoma State University, Stillwater, OK 74078}
\email{jiahong@math.okstate.edu}

\usepackage[margin=1in]{geometry}
\usepackage{amsmath, amsthm, amssymb}
\usepackage{times}

\usepackage{color}
\usepackage{hyperref}

\theoremstyle{plain}
\newtheorem{theorem}{Theorem}[section]
\newtheorem*{definition}{Definition}
\newtheorem{lemma}[theorem]{Lemma}
\newtheorem{proposition}[theorem]{Proposition}

\theoremstyle{definition}

\numberwithin{equation}{section}

\renewcommand{\tilde}{\widetilde}

\def\phi{\varphi}

\newcommand{\be}{\begin{equation}}
\newcommand{\eeq}{\end{equation}}

\newcommand\comb[1]{ {1/2 \choose #1} }

\newcommand{\pa}{\partial}
\newcommand{\la}{\label}
\newcommand{\fr}{\frac}
\newcommand{\na}{\nabla}

\def\RR{{\mathbb R}}
\def\NN{{\mathbb N}}

\def\uu{\boldsymbol{u}}
\def\xx{\boldsymbol{x}}
\def\aa{\boldsymbol{a}}
\def\bb{\boldsymbol{b}}
\def\om{\boldsymbol{\omega}}
\def\cc{\boldsymbol{c}}
\def\ee{\boldsymbol{e}}
\def\kk{\boldsymbol{k}}
\def\alp{\boldsymbol{\alpha}}
\def\bet{\boldsymbol{\beta}}
\def\gam{\boldsymbol{\gamma}}
\def\yy{\boldsymbol{y}}
\def\gg{\boldsymbol{g}}
\def\ll{{\ell}}
\def\KK{\boldsymbol{K}}
\def\XX{\boldsymbol{X}}
\def\AA{\boldsymbol{A}}
\def\II{\boldsymbol{I}}

\begin{document}


\begin{abstract}
We discuss general incompressible inviscid models, including the Euler equations, the surface quasi-geostrophic equation, incompressible porous medium equation, and Boussinesq equations. All these models have classical unique solutions, at least for short time. We show that they have real analytic Lagrangian paths. More precisely, we show that as long as a solution of any of these equations is in a class of regularity that assures H\"{o}lder continuous gradients of velocity, the corresponding Lagrangian paths are real analytic functions of time. The method of proof is conceptually straightforward and general, and we address the combinatorial issues head-on.
\hfill \today.
\end{abstract}

\maketitle

\section{Introduction}
\label{sec:intro}

Analyticity of Lagrangian paths of solutions of incompressible Euler equations is a classical subject. Propagation of real analyticity in space and time, from analytic initial data, and for as long as the solution exists, has been amply investigated~\cite{BardosBenachourZerner76,BardosBenachour77,AlinhacMetivier86,Delort85,LeBail86,LevermoreOliver97,KukavicaVicol09,KukavicaVicol11b,Zheligovsky11,Sawada13}. The smoothness or real analyticity of Lagrangian paths without having analytic Eulerian data is quite a different subject from propagation of analyticity. This subject has been addressed in the past~\cite{Lichtenstein25,Chemin92,Gamblin94,Serfati95a,Serfati95b,Serfati95c,Chemin98,Kato00}, and has recently generated renewed interest~{\cite{Sueur11,GlassSueurTakahashi12,Shnirelman12,FrischZheligovsky12,Isett13,Nadirashvili13,FrischZheligovsky13,FrischVillone14}}. The remarkable property of smoothness of the Lagrangian paths in this system holds even when the Eulerian variables (velocity, pressure) have a limited degree of smoothness. A relatively low degree of smoothness of the Eulerian variables is maintained through the evolution if it is initially present, because the equations, when well posed, are time-reversible. Consequently, the real analyticity of Lagrangian paths in such circumstances is all the more remarkable. An interesting example of the distinct degrees of smoothness of Eulerian and Lagrangian variables is provided in the recent works~\cite{Isett12,DeLellisSzekelyhidiBuckmaster13}, which concern a rough enough Eulerian setting  for non-uniqueness. The purpose of this paper is to show that the real analyticity of Lagrangian paths of solutions of
hydrodynamic models is a general property which occurs naturally when the Eulerian velocities are slightly smoother than Lipschitz, and follows from a uniform chord-arc property of the paths using singular integral calculus.

The Lagrangian paths of any fluid model with velocities $\uu(\xx,t)$, with $\xx\in \RR^d$ and
$t\in \RR$ are defined by ordinary differential equations
\begin{align}
&\fr{d\XX}{dt} = \uu(\XX,t),\la{xeq}\\
&\XX(\aa,0) =\aa.
\end{align}
We refer to $\aa\in \RR^d$ as a ``label'' because it marks the initial point on the path $\aa\mapsto \XX(\aa,t)$. The gradient of the path obeys
\be
\fr{d}{dt}(\na \XX) = (\na \uu)(\na\XX)
\la{nexeq}
\eeq
with initial data the identity matrix. As long as $\uu$ is Lipschitz, we have
\be
\sup_{\aa\in \RR^d}|\na\XX(\aa,t)| \le \exp{\int_0^t \|\na \uu\|_{L^{\infty}}dt }
\la{naxup}
\eeq
where we denote by $\left |\cdot \right |$ the norm of the matrix. The maps $\XX$ are $C^{1,\gamma}$ and invertible if $\uu$ is in $L^1(0,T; C^{1,\gamma})$, and the inverse, the ``back-to-labels'' map $\AA(\xx,t) = \XX^{-1}(\xx,t)$ obeys
\be
\pa_t \AA + u\cdot\na \AA = 0,
\la{aeq}
\eeq
with initial data $\AA(\xx,0)=\xx$. Incompressibility is not needed for this to hold. The gradients obey 
\be
\pa_t (\na \AA) + \uu\cdot\na (\na \AA) + (\na \AA)(\na \uu) = 0,
\la{gradaeq}
\eeq
with initial data the identity matrix, and with $(\na \AA)(\na \uu)$ the matrix product. Therefore
\be
\sup_{\xx\in \RR^d}|\na \AA(\xx,t)| \le \exp{\int_0^t \|\na \uu\|_{L^{\infty}}dt}
\la{naxdown}
\eeq
follows by integrating on characteristics. Because
\[
\aa-\bb = \AA(\XX(\aa,t),t) - \AA(\XX(\bb,t),t)
\]
it follows from \eqref{naxdown} that
\[
|\aa -\bb| \le |\XX(\aa,t)-\XX(\bb,t)|\exp{\int_0^t \|\na \uu\|_{L^{\infty}}dt},
\]
and because
\[
\XX(\aa,t)-\XX(\bb,t) = \int_0^1 \fr{d}{ds}\XX((1-s)\aa +s\bb, t)ds
\]
it follows from \eqref{naxup} that
\[
|\XX(\aa,t) -\XX(\bb,t)| \le |\aa -\bb|\exp{\int_0^t \|\na \uu\|_{L^{\infty}}dt}.
\]
We have thus the chord-arc condition
\be
\lambda^{-1}\le \fr{|\aa-\bb|}{|\XX(\aa,t)-\XX(\bb,t)|} \le \lambda
\la{ac}
\eeq
where
\be
\lambda = \exp{\int_0^t \|\na \uu\|_{L^{\infty}}dt}.
\la{lambda}
\eeq
This condition holds for any fluid system, as long as the velocities are Lipschitz, even if the fluid is compressible. Time analyticity of paths will be discussed here only in the incompressible case, for convenience, but the proofs are the same for compressible equations, modulo differentiating the Jacobian of the path map. 

We consider here one of the following equations: the 2D surface quasi-geostrophic equation (cf.~\eqref{eq:SQG:1}--\eqref{eq:SQG:2}), the 2D incompressible porous medium equation (cf.~\eqref{eq:IPM:1}--\eqref{eq:IPM:2}), the 2D and the 3D incompressible Euler equations (cf.~\eqref{eq:2D:Euler} and \eqref{eq:3D:Euler}), and the 2D Boussinesq equations (cf.~\eqref{eq:Boussinesq:1}--\eqref{eq:Boussinesq:2}). These are by no means an exhaustive list of equations for which our method applies. They have been chosen because, with the sole exception of the 2D Euler equations, all the above models are examples of equations where the question of global existence of smooth solutions remains open. Nevertheless, they all have real analytic particle paths. 
The main result of this manuscript is:
\begin{theorem}[\bf Lagrangian analyticity in hydrodynamic equations]
\label{thm:main}
Consider any of the above hydrodynamic systems on a time interval when the Eulerian velocities are $C^{1,\gamma}$, for some $\gamma \in (0,1)$. Then, as  the chord-arc parameter in \eqref{lambda} remains finite on the time interval, the Lagrangian particle trajectories are real analytic functions of time.
\end{theorem}
We note that the assumption of the theorem holds for short time if the initial data are such that the Eulerian velocities are $C^{1,\gamma}$. The analyticity is a local property. It follows {from} the proof of the theorem that the radius of time analyticity of $\XX(\cdot,t)$ is a function of a suitable norm of the initial data and time, which enters only through the chord-arc parameter $\lambda$. This parameter dependence is consistent with that for the spatial analyticity radius in the case of real analytic initial datum~\cite{KukavicaVicol09,KukavicaVicol11b}.

The main idea of the proof starts with a representation of the velocity in Lagrangian variables in terms of conserved quantities. It is easiest to show this in the case of 2D active scalars. Two dimensional incompressible hydrodynamic velocities can be expressed in terms of a stream function $\psi$,
\be
\uu = \na^{\perp}\psi
\la{upsi}
\eeq
where $\na^\perp = (-\pa_2, \pa_1)$ is the gradient rotated counter-clockwise by 90 degrees. The active scalars solve transport equations
\be
\pa_t \theta + \uu\cdot\na \theta = 0
\la{aseq}
\eeq
with $\uu$ given by \eqref{upsi} and $\psi$ related to $\theta$ by some time independent linear constitutive law $\psi = L\theta$. In most cases this leads to a simple integral formula
\[
\uu(\xx,t) = p.v. \int_{\RR^2} \KK(\xx-\yy)\theta (\yy,t)d\yy
\]
with a kernel $\KK$ that is singular at the origin,
real analytic away from the origin, and integrates to zero on spheres. Note that \eqref{aseq} simply says that
\be
\theta(\XX(\aa,t),t) = \theta_0(\aa).
\la{thetax}
\eeq
Composing the representation of the velocity with the Lagrangian map
we obtain
\be
\fr{d\XX(\aa,t)}{dt} = {\tilde{p.v.}} \int_{\RR^2}\KK(\XX(\aa,t)-\XX(\bb,t))\theta_0(\bb)d\bb
\la{lagex}
\eeq
{where the symbol $\tilde{p.v.}$ denotes a principal value in the Eulerian variables.} Throughout the manuscript, for notational convenience we drop the p.v. in front of the integrals, {as they are always understood as principal values in the Eulerian sense.}
In Section~\ref{sec:examples} we give the precise versions of \eqref{lagex} for the hydrodynamic models under consideration.

The straightforward general idea is to use the chord-arc condition and analyticity of the kernel to prove inductively Cauchy inequalities for all high time derivatives of $\XX$ at fixed label. The implementation of this idea encounters two sets of difficulties: one due to combinatorial complexity, and the other due to the singularity of the kernels and unboundedness of space. 

Combinatorial complexity is already present in a real variables proof of real analyticity of compositions of multivariate real analytic functions. We discuss this issue separately in~Section~\ref{sec:composition}. We use a multivariate Fa\`{a} di Bruno formula (cf.~\cite{ConstantineSavits96} or Lemma~\ref{lem:multi:Faa} below), multivariate identities
(we call them ``magic identities'', because they seem so to us; cf.~Lemma~\ref{lem:multi:magic}) and an induction with modified versions of Cauchy inequalities (cf.~\eqref{eq:X:n} or~\eqref{eq:SQG:TODO}, inspired by~\cite{KrantzParks02}) in order to control the growth of the combinatorial terms. This difficulty is universal, and because we addressed it head-on, the method is applicable to even more examples, not only the ones described in this work, and not only to hydrodynamic ones. 

The singular integral difficulties are familiar. In all these systems the gradient of velocity is also represented using  singular integrals of Calder\'on-Zygmund type. The singular nature of the kernels is always compensated by the presence of polynomial terms in $\XX(\aa,t)-\XX(\bb,t)$, which arise since the kernels have vanishing means on spheres centered at the origin. The fact that we integrate in the whole space necessitates the introduction of a real analytic cutoff, which for simplicity we take to be Gaussian.

The Euler equations have classical invariants~\cite{Constantin01a,Constantin04,FrischZheligovsky13}, which yield completely local relations involving $d\XX/dt$ in Lagrangian coordinates. This is remarkable, but special: in more general systems the corresponding relations are not local. Because of this, we pursue the same proof for the Euler equations as for the general case.

We give the fully detailed proof of Theorem~\ref{thm:main} in the case of the 2D SQG equations. This is done in Section~\ref{sec:SQG}. The proofs for the 2D IPM and 2D and 3D Euler equations are the same. The 2D IPM and 3D Euler equations have of course different kernels;  2D Euler has a less singular kernel. The proof in the case of the 2D Boussinesq equations has an additional level of difficulty since the operator $L$ in the constitutive law for $\theta$ is time-dependent. This issue will be addressed in a forthcoming work.

The paper is organized as follows. In Section~\ref{sec:examples} we provide the self-contained Lagrangian formulae of type \eqref{lagex} for each of the hydrodynamic models under consideration. In Section~\ref{sec:composition} we introduce the combinatorial machinery used in the proof of the main theorem, which is centered around the multivariate Fa\`a di Bruno formula. In Section~\ref{sec:SQG} we give the proof of Theorem~\ref{thm:main} in the case of SQG. Lastly, in Appendix~\ref{sec:Lagrangian:derivation}, for the sake of completeness, we give the derivation of the natural Lagrangian formulae stated in Section~\ref{sec:examples}. In Appendix~\ref{sec:1D} we recall from~\cite{KrantzParks02} the one-dimensional Fa\`a di Bruno formula and its application to the composition of real analytic functions.

\section{Self-contained Lagrangian evolution} 
\label{sec:examples}

In this section we give self-contained formulae for the time derivatives of $\XX$ and $\nabla \XX$, for each of the hydrodynamic equations considered. In each case the initial datum enters these equations as a parameter. 
We use the usual Poisson bracket notation
\[
\{f,g\} = (\pa_1 f)(\pa_2 g) -({\pa_2 f})({\pa_1 g}) = (\na^{\perp}f)\cdot (\na g). 
\]

\subsection{2D Surface Quasi-Geostrophic Equation}
The inviscid SQG  equation is
\begin{align}
&\partial_t\theta + (\uu \cdot \nabla )\theta = 0, \label{eq:SQG:1}\\
&\uu = \nabla^\perp (-\Delta)^{-1/2} \theta= {\mathcal R}^\perp \theta \label{eq:SQG:2}
\end{align}
where ${\mathcal R} = (R_1,R_2)$ is the vector of Riesz-transforms. Here $\xx  \in \RR^2$ and $t>0$.  We recall cf.~\cite{ConstantinMajdaTabak94} that the SQG equation is locally well-posed if $\theta_0\in C^{1,\gamma}$, with $\gamma\in (0,1)$. It follows from \eqref{eq:SQG:1}--\eqref{eq:SQG:2} that the vector fields $\nabla^\perp \theta \cdot \nabla$ and $\partial_t + \uu \cdot \nabla$ commute. The ensuing self-contained formula for the  Lagrangian trajectory $\XX$ induced by the velocity field $\uu$ is
\begin{align}
\frac{d\XX}{dt}(\aa,t) =  \int \KK(\XX(\aa,t)-\XX(\bb,t)) \theta_0(\bb) \, d\bb,
\label{eq:SQG:grad:X:def}
\end{align}
while the gradient of the Lagrangian, $\nabla_a \XX$, obeys 
\begin{align}
&\frac{d(\nabla_a \XX)}{dt} (\aa,t) = \nabla_a\XX(\aa,t)   \int \KK(\XX(\aa,t)-\XX(\bb,t))
\left(\nabla_b^{\perp}{\XX}^{\perp}(\bb,t)\right )\cdot\, \nabla_b \theta_0(\bb)\,d\bb.
\label{eq:SQG:nabla:X:def}
\end{align}
Here the kernel $\KK$ associated to the rotated Riesz transform ${\mathcal R}^\perp$ is given by
\begin{align*}
\KK(\yy) =\frac{\yy^\perp}{2\pi |\yy|^{3}}.
\end{align*}
We refer to Appendix~\ref{sec:app:SQG} for details.

\subsection{The 2D Incompressible Porous Media Equation}
The inviscid IPM equation assumes the form
\begin{align}
&\partial_t \theta  + (\uu\cdot\nabla) \theta =0, \label{eq:IPM:1} \\
&\uu = {\mathbb P} (0,\theta) = -\nabla p - (0,\theta). \label{eq:IPM:2}
\end{align}
We recall, cf.~\cite{CordobaGancedoOrive07} that the IPM equation is locally well-posed if $\theta_0\in C^{1,\gamma}$, with $\gamma\in (0,1)$.
For the particle trajectories $\XX$ induced by the vector field $\uu$ we have
\begin{align*}
\frac{d \XX}{dt}(\aa,t) 
= -\frac1{2\pi} \int \frac{(\XX(\aa,t)-\XX(\bb,t))^\perp}{|(\XX(\aa,t)-\XX(\bb,t)|^2}\,\left\{\theta_0(\bb),X_2(\bb,t) \right\}\,d\bb
\end{align*}
and
\begin{align*}
\frac{d(\nabla_a \XX)}{dt} (\aa,t) 
&=  \nabla_a \XX(\aa,t)  \int \KK(\XX(\aa,t)-\XX(\bb,t))\,   \left\{\theta_0(\bb), X_2(\bb,t)\right\}  \,db  \notag \\
&\qquad + \frac12\left\{\theta_0(\aa), X_2(\aa,t)\right\}\,\left[\begin{array}{cc} 0  & -1\\ 1 & 0\end{array}  \right]\nabla_a \XX(\aa,t)
\end{align*}
where $\KK$ is given by
\begin{align}
\KK(\yy)= \KK(y_1, y_2) = \frac{1}{2\pi |\yy|^4} \left[\begin{array}{cc} 2y_1 \,y_2 & y_2^2-y_1^2\\ y_2^2-y_1^2 & -2y_1 \,y_2\end{array}  \right].
\label{kernelK}
\end{align}
The details are given in Appendix~\ref{sec:app:IPM}.

\subsection{The 3D Euler Equations}
The three-dimensional Euler equations in vorticity form are given by 
\begin{align}
\partial_t \om + \uu \cdot \nabla \om = \om \cdot \nabla \uu 
\label{eq:3D:Euler}
\end{align}
where the divergence free $\uu$ can be recovered from $\om$ via the Biot-Savart formula~\cite{MajdaBertozzi02}
\[
\uu(\xx,t) = \fr{1}{4\pi}\int_{\RR^3}\fr{\xx-\yy}{|\xx-\yy|^3}\times \om(\yy,t)d\yy.
\]
The geometric interpretation of \eqref{eq:3D:Euler} and incompressibility is that the vector fields $\om \cdot \nabla$ and $\partial_t + \uu \cdot \nabla$ commute. The local existence and uniqueness of solutions to \eqref{eq:3D:Euler} with initial data $\uu_0 \in C^{1,\gamma}$, for $\gamma \in (0,1)$, goes back at least to~\cite{Lichtenstein25} (see also~\cite{MajdaBertozzi02} and references therein for a more modern perspective). Due to the Cauchy formula 
\[
\om (\XX(\aa,t), t) = \na\XX(\aa,t)\om_0(\aa),
\]
the Lagrangian map $\XX$ obeys the self-contained evolutions
\[
\frac{d\XX}{dt}(\aa,t) = \frac{1}{4\pi} \int \frac{\XX(\aa,t)-\XX(\bb,t)}{|\XX(\aa,t)-\XX(\bb,t)|^3} \times ( \nabla_b \XX(\bb,t) \om_0(\bb))d\bb
\]
and
\begin{align*}
\frac{d(\nabla_a \XX)}{dt}(\aa,t) &= (\nabla_a \XX)(\aa,t)  \int  \KK(\XX(\aa,t)-\XX(\bb,t)) \left( \nabla_b \XX(\bb,t) \om_0(\bb) \right)  d\yy  \notag\\ 
&\qquad + \frac 12 (\nabla_a\XX(\aa,t)\om_0(\aa)) \times (\nabla_{a} \XX)(\aa,t) 
\end{align*}
where for vectors $\xx$ and $\yy$ the matrix kernel $\KK(\xx)\yy$ is defined in coordinates by
\[
( \KK(\xx) \yy)_{ij}= \frac{3}{8\pi} \frac{\left( \left( \xx \times \yy \right) \otimes \xx + \xx \otimes \left( \xx \times \yy\right)\right)_{ij}}{|\xx|^5} = \frac{3}{8\pi} \frac{\left( \xx \times \yy \right)_i x_j + \left( \xx \times \yy\right)_j x_i}{|\xx|^5}.
\]
The details are given in Appendix~\ref{sec:app:3DE}.

\subsection{The 2D Euler Equations}
The two-dimensional Euler equations in vorticity form are
\begin{align}
\partial_t \omega + \uu \cdot \nabla \omega = 0 \label{eq:2D:Euler}
\end{align}
where the Biot-Savart law~\cite{MajdaBertozzi02} in two dimensions reads
\[
\uu(\xx) = \frac{1}{2\pi} \int \frac{(\xx-\yy)^\perp}{|\xx-\yy|^2} \omega(\yy) d\yy.
\]
The equations are locally in time well-posed  if the initial velocity $\uu_0 \in C^{1,\gamma}$, for some $\gamma \in (0,1)$ (cf.~\cite{Lichtenstein25}). In two dimensions solutions cannot develop finite time singularities~\cite{Yudovich63}, but this fact will not be used in our proof, since global existence is not known for any of the other hydrodynamic equations considered in this paper. The particle trajectory $\XX$ obeys the evolution
\begin{align*} 
\frac{d\XX}{dt}(\aa,t) = \frac{1}{2\pi} \int \frac{(\XX(\aa,t)-\XX(\bb,t))^\perp}{|\XX(\aa,t)-\XX(\bb,t)|^2} \omega_0(\bb) d\bb,
\end{align*}
while the time derivative of $\nabla_a \XX$ obeys
\begin{align*} 
\frac{d(\nabla_a \XX)}{dt}(\aa,t)
&= \nabla_a \XX(\aa,t)  \int \KK(\XX(\aa,t)-\XX(\bb,t))\, \omega_0(\bb) \,d\bb  + \frac12 \omega_0(\aa) \,\left[\begin{array}{cc} 0  & -1\\ 1 & 0\end{array}  \right]\nabla_a \XX(\aa,t)
\end{align*}
with $\KK$ being the kernel in \eqref{kernelK}. These details are given in Appendix~\ref{sec:app:2DE}.

\subsection{The 2D Boussinesq Equations}
The two-dimensional Boussinesq equations for the velocity field $\uu$, scalar pressure $p$, and scalar density $\theta$ are
\begin{align} 
& \partial_t \uu + (\uu \cdot\nabla) \uu = -\nabla p + \theta {\ee}_2, \qquad  \nabla \cdot \uu =0, \label{eq:Boussinesq:1}\\
& \partial_t \theta + (\uu\cdot\nabla) \theta =0, \label{eq:Boussinesq:2} 
\end{align}
where ${\ee}_2 = (0,1)$, $\xx \in \RR^2$, and $t>0$.
The scalar vorticity $\omega =\nabla^\perp \cdot \uu = \partial_{x_1} u_2 - \partial_{x_2} u_1$ satisfies
\begin{align*}
\partial_t \omega + (\uu\cdot\nabla) \omega = \partial_{x_1} \theta.
\end{align*}
The local well-posedness for the 2D Boussinesq holds for initial data
$\uu_0,\theta_0 \in C^{1,\gamma}$ with $\gamma\in (0,1)$ (cf.~\cite{EShu94,ChaeNam97}).
The particle trajectories $\XX$ induced by $\uu$ then obey
\begin{align*}
\frac{d \XX}{dt}(\aa,t) 
&= \frac1{2\pi} \int \frac{(\XX(\aa,t)-\XX(\bb,t))^\perp}{|(\XX(\aa,t)-\XX(\bb,t)|^2}\, \omega_0(\bb) \, d\bb \notag \\
&\quad + \frac1{2\pi} \int \frac{(\XX(\aa,t)-\XX(\bb,t))^\perp}{|(\XX(\aa,t)-\XX(\bb,t)|^2}
\left( \int_0^t \left\{\theta_0(\bb), X_2(\bb,\tau)\right\} d\tau \right) \, d\bb.
\end{align*}
and 
\begin{align*}
\frac{d (\nabla_a \XX)}{dt} (\aa,t) 
&=  \left(  \int \KK(\XX(\aa,t)-\XX(\bb,t))\, \omega_0(\bb) \, d\bb\right) \nabla_a \XX(\aa,t) \notag \\
&\quad +   \left(  \int \KK(\XX(\aa,t)-\XX(\bb,t)) \int_0^t \left\{\theta_0(\bb), X_2(\bb,\tau)\right\} \,d\tau\,d\bb\right)\nabla_a \XX(\aa,t) \notag \\
&\quad + \frac12\left(\omega_0(\aa) + \int_0^t \left\{\theta_0(\aa), X_2(\aa,\tau)\right\}\,d\tau\right)\,\left[\begin{array}{cc} 0  & -1\\ 1 & 0\end{array}  \right]\nabla_a \XX(\aa,t),
\end{align*}
where the kernel $\KK$ is given by \eqref{kernelK}.
The derivation is given in Appendix~\ref{sec:app:Boussinesq}.

\section{Analyticity and the composition of functions: combinatorial lemmas}
\label{sec:composition}

Let $\XX \colon \RR \to \RR^d$ be a vector valued function which obeys the differential equation
\begin{align}
\frac{d}{dt} \XX(t) = \KK(\XX(t))
\label{eq:ODE}
\end{align}
where $\KK\colon \RR^d \to \RR^d$ is a given real analytic function of several variables. In this section we show that if $\XX$ is bounded, then it is in fact real analytic(see Theorem~\ref{thm:dD} below). This statement should be understood in the neighborhood of a point $t_0 \in \RR$, and $\XX_0 = \XX(t_0) \in \RR^d$.

The proof in the case $d=1$ is taken from ~\cite[Chapter 1.5]{KrantzParks02}, and serves as a guiding example (see Appendix~\ref{sec:1D} below). The case $d\geq2$ requires an extended combinatorial machine, and for that we appeal to the multivariate Fa\`a di Bruno formula in~\cite{ConstantineSavits96}. The precise result is:

\begin{theorem}
\label{thm:dD}
 Let $\KK = (K_1,\ldots,K_d) \colon \RR^d \to \RR^d$ be a function which obeys
 \begin{align}
|\partial^{\alp} K_i (\XX) | \leq C \frac{|\alp|!}{R^{|\alp|}}
\label{eq:K:assumption}
\end{align}
for some $C,R>0$, $i \in \{1,\ldots,d\}$, and for all $\XX$ in the neighborhood of some $\XX_{0} = \XX(t_0)$, where
$\XX = (X_1,\ldots,X_d) \colon \RR \to \RR^d$ is a function which obeys
\begin{align}
|X_i(t)|\leq C
\label{eq:X:n=1}
\end{align}
for all $t$ in the neighborhood of $t_0$, and $i \in \{1,\ldots,d\}$. If $\XX$ is a solution of \eqref{eq:ODE}, then we have that
\begin{align}
|(\partial_t^{n} X_i)(t)| \leq (-1)^{n-1} {1/2 \choose n} \frac{(2C)^{n}}{R^{n-1}} n!
\label{eq:X:n}
\end{align}
for all $n\geq 1$,  all coordinates $i \in \{1,\ldots,d\}$, and all $t$ in a neighborhood of $t_0$. In particular, $\XX$ is a real analytic function of $t$ at $t_0$, with radius of analyticity $R/C$.
\end{theorem}

\subsection{Preliminaries}
We denote by $\NN_0$ the set of all integers strictly larger than $-1$, and by $\NN_0^d$ the set of all multi-indices $\alp = (\alpha_1,\cdots,\alpha_d)$ with $\alpha_j \in \NN_0$. For a multi-index $\alp$, we write
\begin{align*}
|\alp| &= \alpha_1 + \ldots + \alpha_d \\
\alp! &=  (\alpha_1!) \cdot \ldots \cdot (\alpha_d !)\\
\partial^{\alp} &= \partial_{x_1}^{\alpha_1} \ldots \partial_{x_d}^{\alpha_d} \\
\yy^{\alp} &= (y_1^{\alpha_1}) \cdot \ldots \cdot (y_d^{\alpha_d})
\end{align*}
where $\yy \in \RR^d$ is a point. The following definition shall be needed below.
\begin{definition}
Let $n\geq 1$, $1\leq s \leq n$, and $\alp \in \NN_0^d$ with $1 \leq |\alp| \leq n$, define the set
\begin{align}
P_s(n,\alp)
&= \Big\{ (\kk_1,\ldots,\kk_s; \ll_1,\ldots,\ll_s) \in \NN_0^d \times\ldots\NN_0^d \times \NN \times \ldots \NN \colon \notag \\
&\qquad 0< |\kk_i| , 0 < \ll_1 < \ldots < \ll_s, \sum_{i=1}^s \kk_i = \alp, \sum_{i=1}^s |\kk_i| \ll_i = n \Big \}.
\label{eq:Ps:def}
\end{align}
In particular, we note that $\ll_i \neq 0$.
\end{definition}
Moreover, for an integer $j\geq 1$ we define
\begin{align*}
{1/2 \choose j} = \frac{(1/2) (1/2-1) \ldots (1/2-j+1)}{j!}
\end{align*}
and 
\[{ 1/2 \choose 0} = -1.\]
We use the above non-standard convention for $\comb{0}$ so that we can ensure
\begin{align*}
(-1)^{j-1} {1/2 \choose j} \geq 0
\end{align*}
for all $j\geq 0$.
Moreover, we will use that
\begin{align}
j! (-1)^{j-1} {1/2 \choose j} =  \frac{1}{2^j} \prod_{k=0}^{j-2} (2k+1) = \frac{(2j-3)!!}{2^j} = \frac{(2j-3)!}{2^{2j-2} (j-2)!} \leq C \frac{j!}{2^j}
\label{eq:factorial:bound}
\end{align}
for some universal constant $C$,
whenever $j\geq 2$.

With this notation in hand, we recall \cite[Theorem 2.1]{ConstantineSavits96}.

\begin{lemma}[\bf Multivariate Fa\`a di Bruno Formula]
\label{lem:multi:Faa}
Let $h\colon \RR^d \to \RR$ be a scalar function, $C^\infty$ in the neighborhood of ${\yy}_{0} = \gg(x_0)$, and $\gg\colon \RR \to \RR^d$ be a vector function, $C^\infty$ in the neighborhood of $x_0$. Define $f(x) = h(\gg(x)) \colon \RR \to \RR$. Then
\begin{align*}
f^{(n)}(x_0) = n! \sum_{1\leq |\alp|\leq n} (\partial^{\alp} h)(\gg(x_0)) \sum_{s=1}^n \sum_{P_s(n,\alp)} \prod_{j=1}^s \frac{\left( (\partial^{\ll_j} \gg)(x_0) \right)^{\kk_j}}{(\kk_j !) (\ll_j !)^{|\kk_j|}}
\end{align*}
holds for any $n\geq 1$, with the convention that $0^0 :=1$.
\end{lemma}

\subsection{Main combinatorial identity}

The following lemma will be essential in the proof of Theorem~\ref{thm:dD}.

\begin{lemma}[\bf Multivaried Magic Identity]
\label{lem:multi:magic}
For $n\geq 1$, with the earlier notation we have that
\begin{align*}
\sum_{1\leq |\alp|\leq n}  (-1)^{|\alp|} |\alp|! \sum_{s=1}^n \sum_{P_s(n,\alp)} \prod_{j=1}^s \frac{{1/2 \choose \ll_j}^{|\kk_j|}}{(\kk_j !)}  = 2 (n+1) {1/2 \choose n+1} .
\end{align*}
\end{lemma}

\begin{proof}[Proof of Lemma~\ref{lem:multi:magic}]
The proof mimics that of the proof of~\cite[Lemma 1.5.2]{KrantzParks02}, by using a diagonal argument.

Let $Z \colon \RR \to \RR$ be defined as
\begin{align*}
Z(t) =  \left( 1 - \sqrt{1- 2t} \right) =  \left(1 - (1+ (-2t) )^{1/2} \right).
\end{align*}
This function has the property that
\begin{align*}
(\partial^{\ll} Z)(0) = - {1/2 \choose \ll} (-2)^{\ll} \ll!
\end{align*}
for any $\ll\geq0$. Also, $Z(0) =0$.

Next, consider a function $K \colon \RR^d \to \RR$, such that
\begin{align*}
(\partial^{\alp} K)(0,\ldots,0) =   |\alp| !
\end{align*}
for any multi-index $\alp \in \NN_0^d$. For example, take a real analytic function of several variables, which on the diagonal is given by
\begin{align*}
K(Z,\ldots,Z) = \frac{1}{1- Z}.
\end{align*}
For example, consider
\begin{align*}
K(Z_1,\ldots,Z_d) = \prod_{j=1}^d \left(\frac{1}{1-Z_j} \right)^{1/d}  
\end{align*}
which is smooth in a neighborhood of the origin in $\RR^d$.

Let $F \colon \RR \to \RR$ be defined as
\begin{align*}
F(t) = K( Z(t),\ldots,Z(t) ) = \frac{1}{\sqrt{1-2t}}.
\end{align*}
This function has the property that
\begin{align}
F^{(n)}(0) = - (n+1)! {1/2 \choose n+1} (-2)^{n+1} \label{eq:F:n:derivatives}
\end{align}
for any $n\geq 1$.

Using Lemma~\ref{lem:multi:Faa} we have on the other hand that
\begin{align*}
F^{(n)}(0)
&= n! \sum_{1\leq |\alp|\leq n} (\partial^{\alp} K)(0,\ldots,0) \sum_{s=1}^n \sum_{P_s(n,\alp)} \prod_{j=1}^s \frac{\left( (\partial^{\ll_j} Z)(0) \right)^{|\kk_j|}}{(\kk_j !) (\ll_j !)^{|\kk_j|}} \notag\\
&= n! \sum_{1\leq |\alp|\leq n} |\alp|! \sum_{s=1}^n \sum_{P_s(n,\alp)} \prod_{j=1}^s \frac{\left( - {1/2 \choose \ll_j} (-2)^{\ll_j} \ll_j! \right)^{|\kk_j|}}{(\kk_j !) (\ll_j !)^{|\kk_j|}}  \notag\\
&= n! (-2)^n \sum_{1\leq |\alp|\leq n} |\alp|! (-1)^{|\alp|} \sum_{s=1}^n \sum_{P_s(n,\alp)} \prod_{j=1}^s \frac{ {1/2 \choose \ll}^{|\kk_j|}}{(\kk_j !) } .
\end{align*}
The proof of the lemma is concluded by appealing to \eqref{eq:F:n:derivatives}.
\end{proof}

\subsection{The proof of Theorem~\ref{thm:dD}}

\begin{proof}[Proof of Theorem~\ref{thm:dD}]
The proof is by induction. The case $n=1$ is contained in assumption \eqref{eq:X:n=1}.

We now show the induction step. Fix one coordinate $i$ throughout the proof. Using the multivaried Fa\`a di Bruno formula of Lemma~\ref{lem:multi:Faa} we obtain
\begin{align*}
(\partial_t^{n+1} X_i) (t)
&= \partial_t^{n} (K_i ( \XX(t)) \notag\\
&=n! \sum_{1\leq |\alp|\leq n} (\partial^{\alp} K_i)(\XX(t)) \sum_{s=1}^n \sum_{P_s(n,\alp)} \prod_{j=1}^s \frac{\left( (\partial_t^{\ll_j} \XX)(t) \right)^{\kk_j}}{(\kk_j !) (\ll_j !)^{|\kk_j|}}.
\end{align*}
By appealing to \eqref{eq:K:assumption} and the inductive hypothesis \eqref{eq:X:n}, we obtain
\begin{align*}
|\partial_t^{n+1} X_i|
&\leq C n! \sum_{1\leq |\alp|\leq n} \frac{|\alp|!}{R^{|\alp|}}   \sum_{s=1}^n \sum_{P_s(n,\alp)} \prod_{j=1}^s \frac{\left( (-1)^{\ll_j-1} {1/2 \choose \ll_j} \frac{(2C)^{\ll_j}}{R^{\ll_j-1}} \ll_j!\right)^{|\kk_j|}}{(\kk_j !) (\ll_j !)^{|\kk_j|}} \notag\\
&\leq C n! (-1)^{n} \frac{(2C)^{n}}{R^n}\sum_{1\leq |\alp|\leq n} (-1)^{|\alp|} |\alp|!   \sum_{s=1}^n \sum_{P_s(n,\alp)} \prod_{j=1}^s \frac{{1/2 \choose \ll_j}^{|\kk_j|}}{(\kk_j !) }\notag\\
&= C n! (-1)^{n} \frac{(2C)^{n}}{R^n} 2 (n+1) {1/2 \choose n+1} \notag\\
& = (-1)^{n} (n+1)! \frac{(2C)^{n+1}}{R^n} {1/2 \choose n+1} n!.
\end{align*}
In the second-to-last inequality we have essentially used Lemma~\ref{lem:multi:magic}. With \eqref{eq:factorial:bound}, the proof is complete.
\end{proof}

\section{Lagrangian analyticity for the SQG equation}
\label{sec:SQG}

In this section we give the proof of Theorem~\ref{thm:main} in the case of the surface quasi-geostrophic equations. The precise statement is:
\begin{theorem}[\bf Lagrangian analyticity for SQG]
Consider initial data $\theta_0 \in C^{1,\gamma} \cap W^{1,1}$, and let $\theta$ be the unique maximal solution of the initial value problem for \eqref{eq:SQG:1}--\eqref{eq:SQG:2}, with $\theta \in L_{\rm loc}^\infty([0,T_*);C^{1,\gamma} \cap W^{1,1})$. Given any $t \in [0,T_*)$, there exists $ T \in (0,  T_* - t)$, with $T = T( \|\nabla u\|_{L^\infty(t,(t+T_*)/2; L^\infty)})$, and $R >0$ with $R= R(t, \|\theta_0\|_{C^{1,\gamma}\cap W^{1,1}},\gamma)$, such that 
\begin{align*}
\|\partial_t^n \XX\|_{L^\infty(t,t+T;C^{1,\gamma})} \leq C n! R^{-n}
\end{align*}
holds for any $n\geq 0$. Here $C$ is a universal constant, and the norm $\|\XX\|_{C^{1,\gamma}}$ is defined in \eqref{eq:X:norm} below. 
In particular, the Lagrangian trajectory $\XX$ is a real analytic function of time, with radius of analyticity $R$.
\end{theorem}

Take any $t \in (0,T_*)$. Analyticity is a local property of functions, so it is sufficient to follow the Lagrangian paths for a short interval of time $[t,t+T]$ past $t$. Note that from the local existence theory we have the bounds on the size of $\theta(\cdot,t)$. Without loss of generality it is sufficient to give the proof for $t=0$.

Fix a $\lambda \in (1,3/2]$ throughout this section. Let $T \in (0,T_*)$ be such that
\begin{align}
\int_{0}^{T} \|\nabla \uu(t)\|_{L^\infty} dt \leq \log \lambda.
\label{eq:small:time}
\end{align}
The existence of this $T$ is a consequence of the local existence theorem. It follows that the chord-arc condition 
\begin{align}
\frac{1}{\lambda} \leq \frac{|\aa-\bb|}{|\XX(\aa,t)-\XX(\bb,t)|} \leq \lambda
\label{eq:X:Lip:bnd}
\end{align}
holds for any $\aa\neq \bb \in \RR^2$ and any $t \in [0,T]$.

For $\gamma \in (0,1)$, define 
\begin{align}
\|\XX\|_{C^{1,\gamma}} := \|\XX(\aa) - \aa \|_{L^\infty} + \|\nabla_a \XX(\aa)  \|_{L^\infty} + [\nabla_a \XX(\aa)]_{C^\gamma}.
\label{eq:X:norm}
\end{align}
Our goal is to use induction in order to show that there exists 
$
C_0 = C_0(\| \theta_0\|_{C^{1,\gamma} \cap W^{1,1}},\gamma,\lambda ) >0
$
and 
$
C_1 = C_1(\lambda,C_K)>0
$
such that the Cauchy inequalities 
\begin{align}
\|\partial_t^{n}\XX\|_{L^\infty(0,T;C^{1,\gamma})} \le (-1)^{n-1}\,n!\, \left(\begin{array}{c} \frac12\\n \end{array}\right) C_0^n C_1^{n-1}
\label{eq:SQG:TODO}
\end{align}
hold for any $n \geq 0$.
Here $\lambda$ is the chord-arc constant in \eqref{eq:X:Lip:bnd}, and $C_K$ is the kernel-dependent constant from \eqref{eq:K:bound} below.

In order to have the induction base case $n=0$ in \eqref{eq:SQG:TODO} taken care of, we choose
\begin{align}
C_0 \geq \| \XX\|_{L^\infty(0,T;C^{1,\gamma})}.
\label{eq:C0:1}
\end{align}
The right side of \eqref{eq:C0:1} is finite in view of the local existence theorem. To prove the induction step, we need to estimate $\sup_{t\in [0,T]} \|\partial_t^{n+1} \XX(\cdot,t)\|_{L^\infty}$,  $\sup_{t\in [0,T]} \| \partial_t^{n+1} (\nabla_a \XX)(\cdot,t)\|_{L^\infty}$, and lastly the H\"older semi norm $ \sup_{t\in [0,T]} [\partial_t^{n+1} (\nabla_a \XX)(\cdot,t)]_{C^\gamma}$. This is achieved in the following three subsections.

\subsection{The \texorpdfstring{$L^\infty$}{L infinity} estimate}
Recall that
\begin{align*}
\frac{d\XX}{dt}(\aa,t) =  \int \KK(\XX(\aa,t)-\XX(\bb,t))
\theta_0(\bb)\, d\bb.
\end{align*}
where $\KK(\yy) = \yy^\perp/(2\pi |\yy|^{3})$. We need to localize this kernel near the origin with a rapidly decaying real analytic function. For this purpose we use a Gaussian and define 
\begin{align*}
\KK_{in} (\yy) =  \frac{\yy^\perp}{2\pi |\yy|^3} e^{-|\yy|^2}
\quad \mbox{and} \quad
\KK_{out} (\yy) =  \frac{\yy^\perp}{2\pi|\yy|^3} (1- e^{-|\yy|^2})
\end{align*}
so that $\KK = \KK_{in} + \KK_{out}$. There exists a universal constant $C_K \geq 1$ such that
\begin{align}
| \partial^{\alp} \KK_{in}(\yy)| \leq  \frac{C_K^{|\alp|} |\alp|!}{|\yy|^{|\alp|+2}} e^{-|\yy|^2/2}
\quad \mbox{and} \quad
| \partial^{\alp} \KK_{out}(\yy)| \leq \frac{C_K^{|\alp|} |\alp|!}{|\yy|^{|\alp|}}
\label{eq:K:bound}
\end{align}
holds for any multi-index $\alp$ and any $\yy\neq 0$. The proof of the above estimates is given in Section~\ref{app:kernel:bound} below. Moreover, since $\int_{\partial B_1(0)} \KK_{in}(\yy) d\yy = 0$, we write
\begin{align}
\frac{d\XX}{dt}(\aa,t) 
&= \int \KK_{in} (\XX(\aa,t)-\XX(\bb,t))
(\theta_0(\bb)-\theta_0(\aa)) \, d\bb\notag\\ 
&\qquad +  \int \KK_{out} (\XX(\aa,t)-\XX(\bb,t))
\theta_0(\bb)  \, d\bb.
\label{eq:SQG:p:0}
\end{align}
We apply $n$ time derivatives to \eqref{eq:SQG:grad:X:def} and obtain
\begin{align}
\partial_t^{n+1}\XX(\aa,t) 
&= \int \partial_t^n \KK_{in}( \XX(\aa,t)-\XX(\bb,t))\,(\theta_0(\bb)-\theta_0(\aa) )\,d\bb \notag\\
&\quad + \int \partial_t^n \KK_{out}( \XX(\aa,t)-\XX(\bb,t)) \theta_0(\bb)\,d\bb
\label{eq:SQG:p:1}
\end{align}
Fix an index $i \in \{1,2\}$ and let either $K = K_{in,i}$ or $K = K_{out,i}$.
Apply the Fa\`a di Bruno formula in Lemma~\ref{lem:multi:Faa} to obtain
\begin{align}
&\partial^{n}_t(K(\XX(\aa,t)-\XX(\bb,t))) \notag\\ 
&\qquad = n! \sum_{1\leq |\alp|\leq n} (\partial^{\alp} K)(\XX(\aa,t)-\XX(\bb,t))
\sum_{s=1}^n \sum_{P_s(n, \alp)} {\prod_{j=1}^{s}}
\frac{(\partial_t^{\ll_j} (\XX(\aa,t)-\XX(\bb,t)))^{\kk_j}}{(\kk_j!)(\ll_j!)^{|\kk_j|}}
\label{eq:SQG:p:2}
\end{align}
Combining formulas \eqref{eq:SQG:p:1} and \eqref{eq:SQG:p:2} with the inductive assumption \eqref{eq:SQG:TODO} for the Lipschitz norm of $\XX$, and the bound \eqref{eq:K:bound}, we arrive at
\begin{align}
|\partial_t^{n+1} \XX(\aa,t)|
& \leq n! \sum_{1\le|\alp|\le n} \int \frac{|\alp|! C_K^{|\alp|} e^{-| \XX(\aa,t)-\XX(\bb,t)|^2/2}}{|\XX(\aa,t)-\XX(\bb,t)|^{2+|\alp|}}  \notag\\
& \qquad \quad \times\sum_{s=1}^n \sum_{P_s(n, \alp)} {\prod_{j=1}^{s}}\,\frac{\left((-1)^{\ll_j-1} \ll_j!{1/2 \choose \ll_j} C_0^{\ll_j} C_1^{\ll_j-1} |\aa-\bb|\right)^{|\kk_j|}}{(\kk_j!)(\ll_j!)^{|\kk_j|}}\, |\theta_0(\bb)-\theta_0(\aa) |\,d\bb \notag\\
& + n! \sum_{1\le|\alp|\le n} \int \frac{|\alp|! C_K^{|\alp|} }{|\XX(\aa,t)-\XX(\bb,t)|^{|\alp|}}  \notag \\
& \qquad \quad \times\sum_{s=1}^n \sum_{P_s(n, \alp)} {\prod_{j=1}^{s}}\,\frac{\left((-1)^{\ll_j-1} \ll_j!{1/2 \choose \ll_j} C_0^{\ll_j} C_1^{\ll_j-1} |\aa-\bb|\right)^{|\kk_j|}}{(\kk_j!)(\ll_j!)^{|\kk_j|}}\, |\theta_0(\bb)|\,d\bb
\label{eq:SQG:p:3}.
\end{align}
From the definition of $P_s(n,\alp)$ in \eqref{eq:Ps:def}, we recall
\begin{align*}
\sum_{j=1}^s \ll_j |\kk_j| =n, \qquad \sum_{j=1}^s |\kk_j| =|\alp|, 
\end{align*}
and estimate \eqref{eq:SQG:p:3} becomes
\begin{align}
|\partial_t^{n+1} \XX(\aa,t)| 
&\leq  n!\,(-1)^n\, C_0^n C_1^n \sum_{1\le|\alp|\le n} (-1)^{|\alp|} |\alp|! C_K^{|\alp|} C_1^{-|\alp|} 
\sum_{s=1}^n \sum_{P_s(n, \alp)} {\prod_{j=1}^{s}} \frac{{1/2 \choose \ll_j}^{|\kk_j|}}{\kk_j!} ( I_{in} + I_{out}) \label{eq:SQG:p:4}
\end{align}
where
\begin{align*}
I_{in} = \int \frac{|\aa-\bb|^{|\alp|} e^{- |\XX(\aa,t) - \XX(\bb,t)|^2 /2}}{|\XX(\aa,t) - \XX(\bb,t)|^{2+|\alp|}} |\theta_0(\bb) - \theta_0(\aa)|d\bb
\end{align*}
and
\begin{align*}
I_{out} = \int \frac{|\aa-\bb|^{|\alp|}}{|\XX(\aa,t) - \XX(\bb,t)|^{|\alp|}} |\theta_0(\bb)| d\bb.
\end{align*}
Using the chord-arc condition \eqref{eq:X:Lip:bnd}, and
\begin{align*}
|\theta_0(\bb)-\theta_0(\aa)| \le [\theta_0]_{C^\gamma}|\aa-\bb|^\gamma,
\end{align*}
we estimate
\begin{align*}
I_{in} \leq [\theta_0]_{C^\gamma} \lambda^{2+|\alp|} \int |\aa-\bb|^{\gamma-2} e^{- |\aa-\bb|^2/(2\lambda^2)} d\bb \leq 8 \lambda^2( \gamma^{-1} +  \lambda ) [\theta_0]_{C^\gamma} \lambda^{|\alp|}.
\end{align*}
On the other hand,  \eqref{eq:X:Lip:bnd} also yields
\begin{align*}
I_{out} \leq \lambda^{|\alp|} \|\theta_0\|_{L^1},
\end{align*}
so that
\begin{align}
I_{in} + I_{out} \leq |\lambda|^{|\alp|} \left(8 \lambda^2( \gamma^{-1} + \lambda ) [\theta_0]_{C^\gamma} + \|\theta_0\|_{L^1} \right)
\label{eq:SQG:p:5}
\end{align}
Therefore, if  we let
\begin{align}
C_1 \geq C_K \lambda
\label{eq:C1:1}
\end{align}
and 
\begin{align}
\frac{C_0}{2} \geq 8 \lambda^2 ( \gamma^{-1} + \lambda ) [\theta_0]_{C^\gamma} + \|\theta_0\|_{L^1} \label{eq:C0:2},
\end{align}
from \eqref{eq:SQG:p:4} and \eqref{eq:SQG:p:5} we conclude
\begin{align}
|\partial_t^{n+1} \XX(\aa,t)| 
& \leq  \frac 12  n!  (-1)^n\, C_0^{n+1} C_1^n \sum_{1\le|\alp|\le n} (-1)^{|\alp|} |\alp|! 
\sum_{s=1}^n \sum_{P_s(n, \alp)} {\prod_{j=1}^{s}} \frac{{1/2 \choose \ll_j}^{|\kk_j|}}{\kk_j!}\notag\\
&\leq  (n+1)!\,(-1)^n {1/2 \choose n+1} C_0^{n+1} C_1^n
\label{eq:SQG:p:6}
\end{align}
where in the last inequality we have appealed to Lemma~\ref{lem:multi:magic}. Estimate \eqref{eq:SQG:p:6} proves the $L^\infty$ portion of the induction step in \eqref{eq:SQG:TODO}.

\subsection{The Lipschitz estimate}
\label{sec:Lip}
Similarly to \eqref{eq:SQG:p:0}, we decompose \eqref{eq:SQG:nabla:X:def} as
\begin{align}
&\frac{d(\nabla_a \XX)}{dt} (\aa,t)  \notag\\
&\ = \nabla_a\XX(\aa,t) \int \KK_{in}(\XX(\aa,t)-\XX(\bb,t)) \left( \nabla_b^\perp \XX^{\perp}(\bb,t)\, \nabla_b \theta_0(\bb) - \nabla_a^\perp \XX^{\perp}(\aa,t)\, \nabla_a \theta_0(\aa) \right) \,d\bb \notag\\
&\qquad + \nabla_a\XX(\aa,t) \int \KK_{out}(\XX(\aa,t)-\XX(\bb,t)) \nabla_b^\perp \XX^{\perp}(\bb,t)\, \nabla_b \theta_0(\bb) \,d\bb.
\label{eq:SQG:g:0}
\end{align}
To estimate the $L^\infty$ norm of $\partial_t^{n+1} (\nabla_a \XX)$, we apply $\partial_t^n$ to \eqref{eq:SQG:g:0}.  By the Leibniz rule we obtain
\begin{align}
&\partial^{n+1}_t \nabla_a \XX (\aa,t) \notag\\
&\quad = \sum_{0\le m\le r \le n} {n \choose r} {r \choose m}
\partial^{n-r}_t\nabla_a \XX(\aa,t) \notag\\
&\qquad \qquad \times \int \partial^m_t \KK_{in}(\XX(\aa,t)-\XX(\bb,t))    \partial_t^{r-m} (\nabla_b^\perp \XX^{\perp}(\bb,t) \nabla_b \theta_0(\bb)-\nabla_a^\perp \XX^{\perp}(\aa,t) \nabla_a \theta_0(\aa)  )\,d\bb \notag\\
&\quad + \sum_{0\le m\le r \le n} {n \choose r} {r \choose m}
\partial^{n-r}_t\nabla_a \XX(\aa,t) \notag\\
&\qquad \qquad \times \int \partial^m_t \KK_{out}(\XX(\aa,t)-\XX(\bb,t))    \partial_t^{r-m} (\nabla_b^\perp \XX^{\perp}(\bb,t)) \nabla_b \theta_0(\bb)\,d\bb.
\label{eq:SQG:g:1}
\end{align}
Invoking the inductive assumption \eqref{eq:SQG:TODO}, we have
\begin{align}
|\partial^{n-r}_t\nabla_a \XX(\aa,t)| \leq (-1)^{n-r-1}\,(n-r)!\, {1/2 \choose n-r}  C_0^{n-r} C_1^{n-r-1}.
\label{eq:SQG:g:2}
\end{align}
Also, in view of \eqref{eq:SQG:TODO} we  estimate
\begin{align}
&|\partial^{r-m}_t(\nabla_b^{{\perp}} \XX^{\perp}(\bb,t)) \nabla_b \theta_0(\bb)- \partial^{r-m}_t(\nabla_a^{{\perp}} \XX^{\perp}(\aa,t)) \nabla_a \theta_0(\aa)| \notag \\
& \qquad\qquad \qquad \le (-1)^{r-m-1}\,(r-m)!\, {1/2 \choose r-m} C_0^{r-m} C_1^{r-m-1}  |\aa-\bb|^\gamma \|\nabla \theta_0\|_{C^\gamma}
\label{eq:SQG:g:3}
\end{align}
and
\begin{align}
| \partial_t^{r-m} (\nabla_b^{{\perp}} \XX^{\perp}(\bb,t)) \nabla_b\theta_0(\bb)| \leq (-1)^{r-m-1}\,(r-m)!\, {1/2 \choose r-m} C_0^{r-m} C_1^{r-m-1} |\nabla_b \theta_0(\bb)|.
\label{eq:SQG:g:4}
\end{align}
Let $i \in \{1,2\}$. Using \eqref{eq:SQG:p:2} and \eqref{eq:K:bound} we  bound
\begin{align}
&|\partial_t^m K_{in,i}(\XX(\aa,t)-\XX(\bb,t))| \notag\\
&\quad \leq m! \sum_{1\leq |\alp|\leq m} 
\frac{C_K^{|\alp|} |\alp|! e^{-|\XX(\aa,t)-\XX(\bb,t)|^2/2}
}{|\XX(\aa,t)-\XX(\bb,t)|^{|\alp|+2}} 
\sum_{s=1}^m \sum_{P_s(m, \alp)} {\prod_{j=1}^{s}}
\frac{( \|\partial_t^{\ll_j} \nabla \XX(\cdot,t) \|_{L^\infty} |\aa-\bb|)^{\kk_j}}{(\kk_j!)(\ll_j!)^{|\kk_j|}} \notag\\
&\quad \leq m!   \sum_{1\leq |\alp|\leq m} 
C_K^{|\alp|} |\alp|! T_{in}
\sum_{s=1}^m \sum_{P_s(m, \alp)} {\prod_{j=1}^{s}}
\frac{\left((-1)^{\ll_j-1} \ll_j!{1/2 \choose \ll_j} C_0^{\ll_j} C_1^{\ll_j-1} \right)^{|\kk_j|}}{(\kk_j!)(\ll_j!)^{|\kk_j|}}\notag\\
&\quad \leq (-1)^m m! C_0^m C_1^m  \sum_{1\leq |\alp|\leq m} 
(-1)^{|\alp|} C_K^{|\alp|} C_1^{-|\alp|} |\alp|! T_{in}
\sum_{s=1}^m \sum_{P_s(m, \alp)} {\prod_{j=1}^{s}} 
\frac{{1/2 \choose \ll_j}^{|\kk_j|}}{\kk_j!}
\label{eq:SQG:g:5:00}
\end{align}
where
\begin{align*}
T_{in} =  \frac{|\aa-\bb|^{|\alp|} e^{-|\XX(\aa,t)-\XX(\bb,t)|^2/2}}{|\XX(\aa,t)-\XX(\bb,t)|^{|\alp|+2}}.
\end{align*}
Using the chord-arc condition \eqref{eq:X:Lip:bnd} we arrive at
\begin{align*}
T_{in} \leq |\aa-\bb|^{-2} e^{-|\aa-\bb|^2/(2\lambda^2)} \lambda^{|\alp|+2} 
\end{align*}
and recalling that $C_1 \geq \lambda C_K$, we obtain from \eqref{eq:SQG:g:5:00} that
\begin{align}
&|\partial_t^m K_{in,i}(\XX(\aa,t)-\XX(\bb,t))| \notag\\
&\quad \leq (-1)^m m! C_0^m C_1^m |\aa-\bb|^{-2} e^{-|\aa-\bb|^2/(2\lambda^2)}  \lambda^2 \sum_{1\leq |\alp|\leq m} 
(-1)^{|\alp|} |\alp|! 
\sum_{s=1}^m \sum_{P_s(m, \alp)} \sum_{j=1}^{s} 
\frac{{1/2 \choose \ll_j}^{|\kk_j|}}{\kk_j!} \notag\\
&\quad \leq (-1)^m m! C_0^m C_1^m |\aa-\bb|^{-2} e^{-|\aa-\bb|^2/(2\lambda^2)}  \lambda^2 2 (m+1) {1/2 \choose m+1}\label{eq:SQG:g:5:01}
\end{align}
where in the last equality we have appealed to Lemma~\ref{lem:multi:magic}.
Similarly, from \eqref{eq:SQG:p:2} and \eqref{eq:K:bound} we have
\begin{align}
&|\partial_t^m K_{out,i}(\XX(\aa,t)-\XX(\bb,t))| \notag\\
&\quad \leq (-1)^m m! C_0^m C_1^m  \sum_{1\leq |\alp|\leq m} 
(-1)^{|\alp|} C_K^{|\alp|} C_1^{-|\alp|} |\alp|! T_{out}
\sum_{s=1}^m \sum_{P_s(m, \alp)} {\prod_{j=1}^{s}}
\frac{{1/2 \choose \ll_j}^{|\kk_j|}}{\kk_j!}.
\label{eq:SQG:g:5:10}
\end{align}
Using \eqref{eq:X:Lip:bnd} we arrive at
\begin{align*}
T_{out} =  \frac{|\aa-\bb|^{|\alp|}}{|\XX(\aa,t)-\XX(\bb,t)|^{|\alp|}} \leq \lambda^{|\alpha|}.
\end{align*}
Therefore, appealing to Lemma~\ref{lem:multi:magic} we arrive at
\begin{align}
|\partial_t^m K_{out,i}(\XX(\aa,t)-\XX(\bb,t))|  \leq (-1)^m m! C_0^m C_1^m 2 (m+1) {1/2 \choose m+1}.
\label{eq:SQG:g:5:11}
\end{align}

Combining \eqref{eq:SQG:g:1}--\eqref{eq:SQG:g:4}, \eqref{eq:SQG:g:5:01}, and \eqref{eq:SQG:g:5:11}, we arrive at
\begin{align}
&|\partial^{n+1}_t \nabla_a \XX (\aa,t)| \notag\\
&\quad \leq I  \sum_{0\le m\le r \le n} {n \choose r} {r \choose m}
(-1)^{n-r-1}\,(n-r)!\, {1/2 \choose n-r}  C_0^{n-r} C_1^{n-r-1} \notag\\
&\qquad \qquad   \times  (-1)^m m! C_0^m C_1^m 2 (m+1) {1/2 \choose m+1}\,  (-1)^{r-m-1}\,(r-m)!\, {1/2 \choose r-m} C_0^{r-m} C_1^{r-m-1} 
\label{eq:SQG:g:6}
\end{align}
where
\begin{align}
I 
&= \lambda^2 \|\nabla \theta_0\|_{C^\gamma} \int |\aa-\bb|^{\gamma-2} e^{-|\aa-\bb|^2/(2\lambda^2)} d\bb + \int |\nabla_b \theta_0 (\bb)| d\bb \notag\\
&\leq 8( \gamma^{-1} +\lambda) \lambda^2 \|\nabla \theta_0\|_{C^\gamma} + \|\nabla \theta_0\|_{L^1} \notag\\
&\leq \frac 18 {C_1^2} C_0 
\label{eq:C0:3}
\end{align}
by making $C_0$ sufficiently large, depending on the initial data.
The above and \eqref{eq:SQG:g:6} imply
\begin{align}
&|\partial^{n+1}_t \nabla_a \XX (\aa,t)| \notag\\
&\quad \leq  \frac 14 C_0^{n+1} C_1^{n}  n! \sum_{0\le m\le r \le n} 
(-1)^{n-r-1}\, {1/2 \choose n-r}\,  (-1)^m   (m+1) {1/2 \choose m+1}\,  (-1)^{r-m-1}\, {1/2 \choose r-m} 
\label{eq:SQG:g:7}
\end{align}
At this stage we invoke another combinatorial identity.
\begin{lemma} \label{lem:combinatorial:new}
We have that
\begin{align}
\sum_{0\le m\le r\le n} (m+1)\,(-1)^m{1/2 \choose m+1} (-1)^{r-m-1} {1/2 \choose r-m} 
(-1)^{n-r-1} {1/2 \choose n-r} \leq 4 (n+1) (-1)^{n} {1/2 \choose n+1}
\label{eq:combinatorial:new}
\end{align}
holds for any integer $n\geq 1$.
\end{lemma}
The proof of Lemma~\ref{lem:combinatorial:new} is given in Section~\ref{sec:proof:combinatorial:lemma} below. From \eqref{eq:SQG:g:7} and \eqref{eq:combinatorial:new} we conclude
\begin{align*}
|\partial^{n+1}_t \nabla_a \XX (\aa,t)|  
\leq  C_0^{n+1} C_1^{n} \frac{n!}{2}   \,(-1)^{n-1} {1/2 \choose n}
\leq C_0^{n+1} C_1^n (-1)^n (n+1)!  {1/2 \choose n+1}
\end{align*}
which concludes the proof of the Lipschitz estimate in the induction step for \eqref{eq:SQG:TODO}.

\subsection{The H\"older estimate for \texorpdfstring{$\nabla_a \XX$}{Grad X}}
{In order to} prove that $[\partial^{n+1}_t \nabla X (a,t)]_{C^\gamma}$ obeys the bound \eqref{eq:SQG:TODO}, we consider the difference
\begin{align*}
\partial^{n+1}_t \nabla \XX (\aa,t) - \partial^{n+1}_t \nabla \XX (\bb,t)
\end{align*}
and estimate it in a similar fashion to $|\partial^{n+1}_t \nabla_a \XX (\aa,t)|$. However, before applying $n$ time derivatives, we use \eqref{eq:SQG:g:0} to re-write
\begin{align}
&\frac{d}{dt} \left( \nabla \XX(\aa,t) - \nabla \XX(\bb,t) \right) \notag\\
&= (\nabla \XX(\aa,t) - \nabla \XX(\bb,t) ) \int \KK_{in} (\XX(\aa,t) - \XX(\cc,t)) (\nabla^\perp \XX^\perp(\cc,t)  \nabla \theta_0 (\cc) - \nabla^\perp \XX^\perp(\aa,t) \nabla\theta_0(\aa) ) d\cc \notag\\
&\ +(\nabla \XX(\aa,t) - \nabla \XX(\bb,t) ) \int \KK_{out} (\XX(\aa,t) - \XX(\cc,t)) \nabla^\perp \XX^\perp(\cc,t) \nabla \theta_0 (\cc) d\cc \notag\\
&\ +{\nabla \XX(\bb,t)  \int  \Big[ \KK_{in}(\XX(\aa,t)-\XX(\cc,t)) \Big(\nabla^\perp \XX^\perp(\cc,t)  \nabla \theta_0 (\cc) - \nabla^\perp \XX^\perp(\aa,t)  \nabla\theta_0(\aa)  \Big) } \notag\\
&\qquad \qquad \qquad \qquad \qquad {- \KK_{in}(\XX(\bb,t)-\XX(\cc,t))   \Big(  \nabla^\perp \XX^\perp(\cc,t)  \nabla \theta_0 (\cc) - \nabla^\perp \XX^\perp(\bb,t) \nabla\theta_0(\bb) \Big) \Big] d\cc } \notag\\
&\ + \nabla \XX(\bb,t) \int \Big(\KK_{out}(\XX(\aa,t)-\XX(\cc,t)) - \KK_{out}(\XX(\bb,t)-\XX(\cc,t)) \Big) \nabla^\perp \XX^\perp(\cc,t)  \nabla \theta_0(\cc) d\cc.
\label{eq:SQG:h:0}
\end{align}
In view of \eqref{eq:SQG:h:0}, similarly to \eqref{eq:SQG:g:1} we write 
\begin{align*}
&\partial^{n+1}_t \nabla \XX (\aa,t) - \partial^{n+1}_t \nabla \XX (\bb,t) = L_1 + L_2 +  L_3 + L_4,
\end{align*}
where
\begin{align}
L_1 &= \sum_{0\le m\le r \le n} {n \choose r} {r \choose m}
\left( \partial^{n-r}_t\nabla \XX(\aa,t) -\partial^{n-r}_t\nabla \XX(\bb,t)\right) \notag\\
&\qquad  \times \int \partial^m_t \KK_{in}(\XX(\aa,t)-\XX(\cc,t))    \partial_t^{r-m} (\nabla^\perp \XX^\perp(\cc,t)  \nabla \theta_0(\cc)-\nabla^\perp \XX^\perp(\aa,t)  \nabla \theta_0(\aa)  )\,d\cc
\label{eq:SQG:h:1} \\
L_2 &= \sum_{0\le m\le r \le n} {n \choose r} {r \choose m}
\left( \partial^{n-r}_t\nabla \XX(\aa,t) - \partial^{n-r}_t\nabla \XX(\bb,t) \right) \notag\\
&\qquad  \times \int \partial^m_t \KK_{out}(\XX(\aa,t)-\XX(\cc,t))    \partial_t^{r-m} (\nabla^\perp \XX^\perp(\cc,t) ) \nabla \theta_0(\cc)\,d\cc 
\label{eq:SQG:h:2} \\
L_3&= \frac 12 \sum_{0\le m\le r \le n} {n \choose r} {r \choose m}
\partial^{n-r}_t\nabla \XX(\bb,t) \notag\\
&\qquad  \times { \int \Big[   \partial^m_t \KK_{in}(\XX(\aa,t)-\XX(\cc,t)) \partial_t^{r-m} \Big(\nabla^\perp \XX^\perp(\cc,t)  \nabla \theta_0(\cc)-\nabla^\perp \XX^\perp(\aa,t)  \nabla \theta_0(\aa)\Big) }\notag \\
&\qquad \quad \ {-  \partial^m_t \KK_{in}(\XX(\bb,t)-\XX(\cc,t))  \partial_t^{r-m} \Big( \nabla^\perp \XX^\perp(\cc,t) \nabla \theta_0(\cc)-\nabla^\perp \XX^\perp(\bb,t) \nabla \theta_0(\bb)  \Big) \Big] d\cc }
\label{eq:SQG:h:3}\\
L_4&=  \sum_{0\le m\le r \le n} {n \choose r} {r \choose m}
\partial^{n-r}_t\nabla \XX(\bb,t) \notag\\
&\qquad  \times \int \Big( \partial^m_t \KK_{out}(\XX(\aa,t)-\XX(\cc,t))  - \partial^m_t \KK_{out}(\XX(\bb,t)-\XX(\cc,t))  \Big) \notag\\
&\qquad  \qquad \times
 \partial_t^{r-m} (\nabla^\perp \XX^\perp(\cc,t) ) \nabla \theta_0(\cc) d\cc.
 \label{eq:SQG:h:4}
\end{align}
First we notice that by using the bound 
\begin{align*}
|\partial^{n-r}_t\nabla^\perp \XX^\perp(\aa,t) -\partial^{n-r}_t\nabla^\perp \XX^\perp(\bb,t) | \leq |\aa-\bb|^\gamma (-1)^{n-r-1} (n-r)! {1/2 \choose n-r} C_0^{n-r} C_1^{n-r-1}
\end{align*}
instead of \eqref{eq:SQG:g:2}, precisely as in Section~\ref{sec:Lip} above we show that
\begin{align}
L_1 + L_2 \leq \frac 12 |\aa - \bb|^\gamma C_0^{n+1} C_1^n (-1)^n (n+1)!  {1/2 \choose n+1}
\label{eq:SQG:h:5}
\end{align}
under precisely the same conditions on $C_0$ and $C_1$ as above. 

In order to estimate $L_3$, we decompose it as
\begin{align*}
L_3 = L_{31} + L_{32} + L_{33} + L_{34},
\end{align*}
where
\begin{align*}
L_{31}&= \sum_{0\le m\le r \le n} {n \choose r} {r \choose m}
\partial^{n-r}_t\nabla \XX(\bb,t) \notag\\
&\  \times \int_{|\cc - \frac{\aa+\bb}{2}| \leq 4 |\aa-\bb|}  \partial^m_t \KK_{in}(\XX(\aa,t)-\XX(\cc,t))    \partial_t^{r-m} (\nabla^\perp \XX^\perp(\cc,t)  \nabla \theta_0(\cc)-\nabla^\perp \XX^\perp(\aa,t)  \nabla \theta_0(\aa)  )  d\cc\\
L_{32}&= - \sum_{0\le m\le r \le n} {n \choose r} {r \choose m}
\partial^{n-r}_t\nabla \XX(\bb,t) \notag\\
&\  \times \int_{|\cc - \frac{\aa+\bb}{2}| \leq 4 |\aa-\bb|}  \partial^m_t \KK_{in}(\XX(\bb,t)-\XX(\cc,t))    \partial_t^{r-m} (\nabla^\perp \XX^\perp(\cc,t)  \nabla \theta_0(\cc)-\nabla^\perp \XX^\perp(\bb,t)  \nabla \theta_0(\bb)  )  d\cc
\end{align*}
account for the singular pieces, and  
\begin{align*}
L_{33}&=  \frac 12  \sum_{0\le m\le r \le n} {n \choose r} {r \choose m}
\partial^{n-r}_t\nabla \XX(\bb,t) \Big(\XX(\aa,t) - \XX(\bb,t)\Big) \notag\\
&\qquad \times \int_{|\cc - \frac{\aa+\bb}{2}| \geq 4 |\aa-\bb|} \int_0^1 \partial_t^m \nabla \KK_{in}( \rho \XX(\aa,t) + (1-\rho) \XX(\bb,t) - \XX(\cc,t)) d\rho \notag \\
&\qquad  \qquad \qquad \qquad\times {\partial_t^{r-m} \Big(2 \nabla^\perp \XX^\perp(\cc,t) \nabla \theta_0(\cc)-\nabla^\perp \XX^\perp(\aa,t) \nabla \theta_0(\aa) -\nabla^\perp \XX^\perp(\bb,t) \nabla \theta_0(\bb)   \Big) } d\cc\\
L_{34}&=  \frac 12  \sum_{0\le m\le r \le n} {n \choose r} {r \choose m}
\partial^{n-r}_t\nabla \XX(\bb,t) {\partial_t^{r-m} \Big( \nabla^\perp \XX^\perp(\bb,t) \nabla \theta_0(\bb) -\nabla^\perp \XX^\perp(\aa,t)  \nabla \theta_0(\aa)   \Big) }  \notag\\
&\qquad \times \int_{|\cc - \frac{\aa+\bb}{2}| \geq 4 |\aa-\bb|}  \Big( \partial_t^m  \KK_{in}( \XX(\aa,t) - \XX(\cc,t)) + \partial_t^m  \KK_{in}( \XX(\bb,t) - \XX(\cc,t)) \Big) d\cc  
\end{align*}
{account for the pieces at infinity.} {Here, we have used the polarization identity $L_{33}+L_{34} =  x_1y_1-x_2y_2= (x_1-x_2)(y_1+y_2)/2 + (x_1+x_2)(y_1-y_2)/2$.} Moreover, for the term $L_{33}$ in the above decomposition we have used the mean value theorem to  write
\begin{align*}
&\partial^m_t \KK_{in}(\XX(\aa,t)-\XX(\cc,t))  -  \partial^m_t \KK_{in}(\XX(\bb,t)-\XX(\cc,t)) \notag\\
&\qquad = (\XX(\aa,t)-\XX(\bb,t)) \cdot \int_0^1 \partial_t^m \nabla \KK_{in}( \rho \XX(\aa,t) + (1-\rho) \XX(\bb,t) - \XX(\cc,t)) d\rho.
\end{align*}
We first bound $L_{31}$ and $L_{32}$. We appeal to \eqref{eq:SQG:g:2}, \eqref{eq:SQG:g:3}, \eqref{eq:SQG:g:5:01},  \eqref{eq:SQG:g:6}, and Lemma~\ref{lem:combinatorial:new} to obtain
\begin{align}
L_{31} + L_{32} 
&\leq  C_0^{n}C_1^{n} (-1)^n (n+1)! {1/2 \choose n+1} I_{3,in} 
\label{eq:SQG:h:7}
\end{align}
where
\begin{align*}
I_{3,in} 
&=\lambda^2 C_1^{-2} \|\nabla \theta_0\|_{C^\gamma} \int_{|\cc - \frac{\aa+\bb}{2}| \leq 4 |\aa-\bb|} |\bb-\cc|^{\gamma-2} e^{-|\bb-\cc|^2/(2\lambda^2)} + |\aa-\cc|^{\gamma-2} e^{-|\aa-\cc|^2/(2\lambda^2)} d\cc \notag\\
&\leq 20 \pi \gamma^{-1} C_K^{-2} \|\nabla \theta_0\|_{C^\gamma} |\aa-\bb|^{\gamma}.
\end{align*}
since $C_1 \geq \lambda C_K$. Letting
\begin{align}
C_0 \geq 160 \pi \gamma^{-1} C_K^{-2} \|\nabla \theta_0\|_{C^\gamma} 
\label{eq:C0:4}
\end{align}
we obtain in combination with \eqref{eq:SQG:h:7} that
\begin{align}
L_{31} + L_{32} 
&\leq \frac 18  |\aa-\bb|^{\gamma} C_0^{n+1}C_1^{n} (-1)^n (n+1)! {1/2 \choose n+1}
\label{eq:SQG:h:8}
\end{align}
holds. In order to estimate $L_{33}$, we notice that due to the chord-arc condition, 
\begin{align*}
| \XX(\bb,t) - \XX(\cc,t) - \bb + \cc| 
&\leq \lambda |\bb-\cc| \int_0^t \|\nabla u(s)\|_{L^\infty} ds \leq \lambda \log \lambda |\bb-\cc|,
\end{align*}
and similarly for $\aa$ and $\cc$.
Thus, we have that 
\begin{align*}
& |\rho \XX(\aa,t) + (1-\rho) \XX(\bb,t) - \XX(\cc,t)| \notag \\
&\ \geq  |\rho \aa + (1-\rho) \bb - \cc| - \rho |\XX(\aa,t) - \aa - \XX(\cc,t) + \cc| - (1-\rho) | \XX(\bb,t)- \bb - \XX(\cc,t) + \cc| \notag\\
&\ \geq |\cc- (\aa+\bb)/2| - |\aa-\bb|/2 - \lambda \log \lambda ( \rho |\aa-\cc| + (1-\rho) |\bb-\cc| ) \notag \\
&\ \geq |\cc- (\aa+\bb)/2| - |\aa-\bb|/2 - \lambda \log \lambda ( |\cc- (\aa+\bb)/2| + |\aa-\bb|/2 )
\end{align*}
holds for any $\rho \in (0,1)$. Therefore, in view of the choice $\lambda \in (1,3/2]$ we have that $\lambda \log \lambda \leq 2/3$, and thus
\begin{align}
 |\rho \XX(\aa,t) + (1-\rho) \XX(\bb,t) - \XX(\cc,t)| \geq |\cc- (\aa+\bb)/2|/3 - |\aa-\bb| \geq |\cc- (\aa+\bb)/2|/12\end{align}
holds whenever  $|\cc - (\aa + \bb)/2| \geq 4 |\aa-\bb|$. Using \eqref{eq:K:bound} and \eqref{eq:SQG:p:2} we thus  bound
\begin{align}
&\int_0^1 | \partial_t^m \nabla \KK_{in}( \rho \XX(\aa,t) + (1-\rho) \XX(\bb,t) - \XX(\cc,t) )| d\rho \notag\\ 
&\quad  
\leq m! \sum_{1\leq |\alp|\leq m}  | \partial^{\alp} \nabla \KK_{in}(\rho \XX(\aa,t) + (1-\rho) \XX(\bb,t) - \XX(\cc,t) )|\notag\\
&\qquad \qquad \qquad \times
\sum_{s=1}^m \sum_{P_s(m, \alp)} \sum_{j=1}^{s} 
\frac{(\rho | \partial_t^{\ll_j} (\XX(\aa,t)-\XX(\cc,t))| +  (1-\rho) |\partial_t^{\ll_j} (\XX(\bb,t)-\XX(\cc,t)|)^{\kk_j}}{(\kk_j!)(\ll_j!)^{|\kk_j|}}  \notag\\
&\quad  
\leq m! \sum_{1\leq |\alp|\leq m} 
\frac{C_K^{|\alp|+1} (12)^{|\alp|+3} (|\alp|+1)! e^{-|\cc-(\aa+\bb)/2|^2/(288)}}{|\cc- (\aa+\bb)/2|^{|\alp|+3}} \notag\\
&\qquad \qquad \qquad \times
\sum_{s=1}^m \sum_{P_s(m, \alp)} \sum_{j=1}^{s} 
\frac{\left(\frac{9 \lambda|\cc - (\aa+\bb)/2|}{8} (-1)^{\ll_j-1} {1/2 \choose \ll_j} C_0^{\ll_j} C_1^{\ll_j -1}\right)^{\kk_j}}{\kk_j!}
\label{eq:SQG:h:9}
\end{align}
Therefore, once we notice that $|\alp| + 1 \leq 2^{|\alp|}$, if we let
\begin{align}
C_1 \geq 27 \lambda C_K,
\label{eq:C1:2}
\end{align}
from \eqref{eq:SQG:h:9} and Lemma~\ref{lem:multi:magic} we deduce that
\begin{align}
&\int_0^1 | \partial_t^m \nabla \KK_{in}( \rho \XX(\aa,t) + (1-\rho) \XX(\bb,t) - \XX(\cc,t) )| d\rho \notag\\ 
&\quad  \leq 2 C_k 12^3 m! (m+1) (-1)^m {1/2 \choose m+1} C_0^m C_1^m \frac{e^{-|\cc-(\aa+\bb)/2|^2/(288)}}{|\cc-(\aa+\bb)/2|^3}
\label{eq:SQG:h:10}.
\end{align}
Using \eqref{eq:SQG:g:2}, \eqref{eq:SQG:g:3}, \eqref{eq:SQG:g:6},  Lemma~\ref{lem:combinatorial:new},  and \eqref{eq:SQG:h:10}, we arrive at
\begin{align}
L_{33}
&\leq C_0^{n}C_1^{n} (-1)^n (n+1)! {1/2 \choose n+1} |\aa-\bb| I_{3,out}
\label{eq:SQG:h:11}
\end{align}
where
\begin{align}
I_{3,out} 
&= 2 \lambda C_K 12^3 C_1^{-2} \| \nabla \theta_0\|_{C^\gamma} \int_{|\cc - (\aa+\bb)/2|\geq 4 |\aa-\bb|} {\frac{ |\aa-\cc|^\gamma + |\bb-\cc|^\gamma}{2} } \frac{e^{-|\cc-(\aa+\bb)/2|^2/(288)}}{|\cc-(\aa+\bb)/2|^3} d\cc \notag\\
&\leq 144 \| \nabla \theta_0\|_{C^\gamma} \int_{|\cc - (\aa+\bb)/2|\geq 4 |\aa-\bb|} |\cc-(\aa+\bb)/2|^{\gamma-3}  d\cc \notag\\
&\leq 288 \pi /(1-\gamma)  \| \nabla \theta_0\|_{C^\gamma} (4 |\aa-\bb|)^{\gamma-1}\notag \\
&\leq \frac{1}{16} C_0 |\aa-\bb|^{\gamma-1} 
\label{eq:C0:5}
\end{align}
if we choose $C_0$ sufficiently large. 
From \eqref{eq:SQG:h:11} and \eqref{eq:C0:5} we conclude that 
\begin{align}
L_{33} \leq  \frac{1}{16} C_0^{n+1}C_1^{n} (-1)^n (n+1)! {1/2 \choose n+1} |\aa-\bb|^\gamma
\label{eq:SQG:h:88}.
\end{align}

{In order to estimate $L_{34}$ we need to appeal to one more cancellation property: each component of the kernel $\KK$ is a derivative of a non-singular scalar kernel, i.e.}
\[
{\KK(\yy) = \frac{\yy^\perp}{2\pi |\yy|^3} = \nabla^\perp_{y} \left( \frac{-1}{2\pi |\yy|} \right).}
\]
{This is in fact the reason why $\KK$ has zero mean on spheres.} {The kernels associated to each of the hydrodynamic systems considered in this paper obey this property.} {The upshot of the above identity is that we have}
\begin{align}
{\KK_{in}(\yy) = \nabla_y^\perp \Big(K_{in}^{(1)}(\yy)\Big) + \KK_{in}^{(2)}(\yy)}  = \{ K_{in}^{(1)}(\yy),\yy \} + \KK_{in}^{(2)}(\yy)
\label{eq:Kin:decomposition}
\end{align}
where $\{\cdot,\cdot\}$ denotes the Poisson bracket, and 
\[
K_{in}^{(1)}(\yy) = \frac{-1}{2\pi |\yy|} e^{-|\yy|^2}
\]
and 
\[
\KK_{in}^{(2)}(\yy) = \frac{-\yy^\perp}{\pi |\yy|} e^{-|\yy|^2}.
\]
Similarly to \eqref{eq:K:bound}, there exits $C_K>0$ such that 
\begin{align}
| \partial^{\alp} K_{in}^{(1)}(\yy)| \leq  \frac{C_K^{|\alp|} |\alp|!}{|\yy|^{|\alp|+1}} e^{-|\yy|^2/2}
\quad \mbox{and} \quad
| \partial^{\alp} \KK_{in}^{(2)}(\yy)| \leq \frac{C_K^{|\alp|} |\alp|!}{|\yy|^{|\alp|}} e^{-|\yy|^2/2}
\label{eq:K:bound:2}
\end{align}
holds for any multi-index $\alp$ and any $\yy\neq 0$. 

The importance of the cancellation property hidden in  \eqref{eq:Kin:decomposition} is seen as follows. When bounding the term $L_{34}$ we need to estimate
\begin{align*} 
T_m(\aa) := \int_{|\cc - \frac{\aa+\bb}{2}| \geq 4 |\aa-\bb|}  \partial_t^m  \KK_{in}( \XX(\aa,t) - \XX(\cc,t))  d\cc,
\end{align*}
and a similarly defined $T_m(\bb)$.
Due to \eqref{eq:Kin:decomposition}, and the change of variables
\begin{align*} 
&(\nabla^\perp_j K_{in}^{(1)}) (\XX(\aa,t)-\XX(\cc,t)) \notag\\
&= - \frac{\partial X_j}{\partial c_2}(\cc,t) \frac{\partial}{\partial c_1} K_{in}^{(1)}(\XX(\aa,t)-\XX(\cc,t)) + \frac{\partial X_j}{\partial c_1}(\cc,t) \frac{\partial}{\partial c_2} K_{in}^{(1)}(\XX(\aa,t)-\XX(\cc,t)) \notag\\
&= - \{ K_{in}^{(1)}(\XX(\aa,t)-\XX(\cc,t)) ,   X_j(\cc,t) \}
\end{align*}
which holds due to the Poisson bracket being invariant under composition with a divergence-free $\XX$,
we rewrite
\begin{align*} 
T_m(\aa) 
&= \int_{|\cc - \frac{\aa+\bb}{2}| \geq 4 |\aa-\bb|}  \partial_t^m  \KK_{in}^{(2)}( \XX(\aa,t) - \XX(\cc,t))  d\cc  \notag\\
&\qquad - \int_{|\cc - \frac{\aa+\bb}{2}| \geq 4 |\aa-\bb|}  \partial_t^m \Big\{ K_{in}^{(1)}(\XX(\aa,t)-\XX(\cc,t)) ,  \XX(\cc,t) \Big\}  d\cc \notag\\
&= \int_{|\cc - \frac{\aa+\bb}{2}| \geq 4 |\aa-\bb|}  \partial_t^m  \KK_{in}^{(2)}( \XX(\aa,t) - \XX(\cc,t))  d\cc  \notag\\
&\qquad - \sum_{i=0}^m {m \choose i} \int_{|\cc - \frac{\aa+\bb}{2}| \geq 4 |\aa-\bb|}   \Big\{ \partial_t^{i} K_{in}^{(1)}(\XX(\aa,t)-\XX(\cc,t)) , \partial_t^{m-i} \XX(\cc,t) \Big\}  d\cc.
\end{align*}
In the second term in the above, we integrate by parts in the $\cc$ variable (the variable in which the derivatives in the Poisson bracket are taken) and note that $\cc$-derivatives commute with $t$-derivatives, to obtain
\begin{align}
T_m(\aa)&= \int_{|\cc - \frac{\aa+\bb}{2}| \geq 4 |\aa-\bb|}  \partial_t^m  \KK_{in}^{(2)}( \XX(\aa,t) - \XX(\cc,t))  d\cc  \notag\\
&\quad - \sum_{i=0}^m {m \choose i} \int_{|\cc - \frac{\aa+\bb}{2}| = 4 |\aa-\bb|} \partial_t^{i} K_{in}^{(1)}(\XX(\aa,t)-\XX(\cc,t)) \,  {\boldsymbol{n}}_c^\perp \cdot  \partial_t^{m-i} (\nabla_c\XX(\cc,t) )  d\sigma(\cc)
\label{eq:Ta:1}
\end{align}
where ${\boldsymbol{n}}$ is the outward unit normal to the circle $\{ \cc \colon |\cc - \frac{\aa+\bb}{2}| = 4 |\aa-\bb| \}$. The corresponding formula also holds for $T_m(\bb)$.
Using \eqref{eq:K:bound:2} and the argument used to prove \eqref{eq:SQG:g:5:01}, it follows that 
\begin{align} 
| \partial_t^i K_{in}^{(1)} (\XX(\aa,t) - \XX(\cc,t)) | \leq (-1)^i i! C_0^i C_1^i \frac{e^{-|\aa-\cc|^2/(2\lambda^2)}}{|\aa-\cc|} \lambda^2 2 (i+1) \comb{i+1}
\label{eq:Kin:1:bound}
\end{align}
for all $i\geq 0$
and
\begin{align} 
| \partial_t^m \KK_{in}^{(2)} (\XX(\aa,t) - \XX(\cc,t)) | \leq (-1)^m m! C_0^m C_1^m e^{-|\aa-\cc|^2/(2\lambda^2)}  \lambda^2 2 (m+1) \comb{m+1}
\label{eq:Kin:2:bound}
\end{align}
for all $m\geq 0$.
Therefore, using \eqref{eq:SQG:g:2} and \eqref{eq:K:bound:2}--\eqref{eq:Kin:2:bound} we conclude that 
\begin{align} 
|T_m(\aa)| 
&\leq (-1)^m m! C_0^m C_1^m \lambda^2 2(m+1) \comb{m+1} \int_{|\cc - \frac{\aa+\bb}{2}| \geq 4 |\aa-\bb|} e^{-|\aa-\cc|^2/(2\lambda^2)} d\cc\\
&\qquad + \sum_{i=0}^m {m \choose i} (-1)^i i! C_0^i C_1^i \lambda^2 2 (i+1) \comb{i+1} 
(-1)^{m-i-1} (m-i)! \comb{m-i} C_0^{m-i} C_1^{m-i-1} \notag\\
&\qquad \qquad \qquad \times \int_{|\cc - \frac{\aa+\bb}{2}| = 4 |\aa-\bb|} \frac{e^{-|\aa-\cc|^2/(2\lambda^2)}}{|\aa-\cc|}  d\sigma(\cc) \notag\\
&\leq \frac{1}{128} (-1)^m m! C_0^{m+1} C_1^m   2(m+1) \comb{m+1}
\label{eq:Ta:2}
\end{align}
by choosing $C_0$ sufficiently large, depending on $\lambda$. Here we have used that $|\aa-\cc|\geq 3 |\aa-\bb|$ and the combinatorial identity
\begin{align*} 
\sum_{i=0}^m 2(i+1) (-1)^i \comb{i+1} (-1)^{m-i-1} \comb{m-i} = 4 (-1)^{m} (m+1) \comb{m+1}
\end{align*}
which is proven using the argument given in Section~\ref{sec:proof:combinatorial:lemma}.
To conclude the $T_{34}$ bound, we combine \eqref{eq:Ta:2} and the corresponding estimate for the $\bb$ term, with \eqref{eq:SQG:g:2}, \eqref{eq:SQG:g:3}, and Lemma~\ref{lem:combinatorial:new} to obtain
\[
L_{34}\leq  \frac{1}{16} C_0^{n+1}C_1^{n} (-1)^n (n+1)! {1/2 \choose n+1} |\aa-\bb|^\gamma
\]
for all $n\geq 0$. 

Thus, from \eqref{eq:SQG:h:8}, \eqref{eq:SQG:h:88}, and  the above estimate for $L_{34}$, we obtain the desired bound for $L_3$, namely
\begin{align}
L_{3}
&\leq \frac 14 C_0^{n+1}C_1^{n} (-1)^n (n+1)! {1/2 \choose n+1} |\aa-\bb|^\gamma.
\label{eq:SQG:h:12}
\end{align}

It is left to estimate $L_4$, as defined in \eqref{eq:SQG:h:4}, which is achieved similarly to $L_3$. First we decompose
\begin{align*}
L_4 = L_{41} + L_{42} + L_{43},
\end{align*}
where
\begin{align*}
L_{41}&= \sum_{0\le m\le r \le n} {n \choose r} {r \choose m}
\partial^{n-r}_t\nabla  \XX(\bb,t) \notag\\
&\  \times \int_{|\cc - \frac{\aa+\bb}{2}| \leq 4 |\aa-\bb|}  \partial^m_t \KK_{out}(\XX(\aa,t)-\XX(\cc,t))    \partial_t^{r-m} (\nabla^\perp \XX^\perp(\cc,t)) \nabla \theta_0(\cc)   d\cc\\
L_{42}&= - \sum_{0\le m\le r \le n} {n \choose r} {r \choose m}
\partial^{n-r}_t\nabla  \XX(\bb,t) \notag\\
&\  \times \int_{|\cc - \frac{\aa+\bb}{2}| \leq 4 |\aa-\bb|}  \partial^m_t \KK_{out}(\XX(\bb,t)-\XX(\cc,t))    \partial_t^{r-m} (\nabla^\perp \XX^\perp(\cc,t) ) \nabla \theta_0(\cc)   d\cc
\end{align*}
and
\begin{align*}
L_{43}&= \sum_{0\le m\le r \le n} {n \choose r} {r \choose m}
\partial^{n-r}_t\nabla  \XX(\bb,t) (\XX(\aa,t) - \XX(\bb,t)) \notag\\
&\qquad \times \int_{|\cc - \frac{\aa+\bb}{2}| \geq 4 |\aa-\bb|} \int_0^1 \partial_t^m \nabla \KK_{out}( \rho \XX(\aa,t) + (1-\rho) \XX(\bb,t) - \XX(\cc,t)) d\rho \notag \\
&\qquad  \qquad \qquad \qquad\times \partial_t^{r-m} (\nabla^\perp \XX^\perp(\cc,t) ) \nabla \theta_0(\cc) d\cc
\end{align*}
We appeal to \eqref{eq:SQG:g:2}, \eqref{eq:SQG:g:4}, \eqref{eq:SQG:g:5:11},  and Lemma~\ref{lem:combinatorial:new} to obtain
\begin{align}
L_{41} + L_{42} 
&\leq  C_0^{n}C_1^{n} (-1)^n (n+1)! {1/2 \choose n+1} I_{4,in} 
\label{eq:SQG:h:13}
\end{align}
under the standing assumptions on $C_0$ and $C_1$, where
\begin{align*}
I_{4,in} 
&= \int_{|\cc - \frac{\aa+\bb}{2}| \leq 4 |\aa-\bb|}  |\nabla \theta_0 (\cc)| d\cc \leq (16 \pi |\aa-\bb|^2)^{\gamma/2} \|\nabla \theta_0\|_{L^{2/(2-\gamma)}} \leq C_0 |\aa-\bb|^\gamma
\end{align*}
by letting 
\begin{align}
C_0 \geq 8 (16 \pi)^{\gamma/2}\left( \|\nabla \theta_0\|_{L^1} +  \|\nabla \theta_0\|_{L^\infty} \right)
\label{eq:C0:6}.
\end{align}
From \eqref{eq:SQG:h:13} and \eqref{eq:C0:6} we obtain the desired bound
\begin{align}
L_{41} + L_{42} 
&\leq  \frac 18 C_0^{n+1}C_1^{n} (-1)^n (n+1)! {1/2 \choose n+1} |\aa-\bb|^\gamma.
\label{eq:SQG:h:14}
\end{align}
Estimating $L_{43}$ is similar to bounding $L_{33}$. First, note that similarly to \eqref{eq:SQG:h:10}, under the standing assumptions on $C_0$ and $C_1$ we have
\begin{align}
&\int_0^1 \partial_t^m \nabla \KK_{in}( \rho \XX(\aa,t) + (1-\rho) \XX(\bb,t) - \XX(\cc,t)) d\rho  \notag\\
&\qquad \leq 24 C_k m! (m+1) (-1)^m {1/2 \choose m+1} C_0^m C_1^m \frac{1}{|\cc-(\aa+\bb)/2|}
\label{eq:SQG:h:14b}
\end{align}
for $|\cc - (\aa+\bb)/2| \geq 4 |\aa-\bb|$. Combining \eqref{eq:SQG:g:2}, \eqref{eq:SQG:g:4},  Lemma~\ref{lem:combinatorial:new}, and \eqref{eq:SQG:h:14b} we obtain
\begin{align}
L_{43} \leq C_0^{n}C_1^{n} (-1)^n (n+1)! {1/2 \choose n+1}  |\aa-\bb|  I_{4,out}
\label{eq:SQG:h:15a}
\end{align}
where
\begin{align}
I_{4,out}
&= 24 \lambda C_k C_1^{-2} \int_{|\cc - \frac{\aa+\bb}{2}| \geq 4 |\aa-\bb|} \frac{|\nabla \theta_0 (\cc)|}{|\cc-(\aa+\bb)/2|}  d\cc  \notag\\
&\leq C_\gamma \|\nabla \theta_0\|_{L^{2/(2-\gamma)}} |\aa-\bb|^{\gamma-1} 
\leq \frac 18 C_0  |\aa-\bb|^{\gamma-1} \label{eq:C0:7}
\end{align}
by choosing $C_0$ sufficiently large. Finally, from \eqref{eq:SQG:h:14}--\eqref{eq:C0:7} we obtain that 
\begin{align}
L_4 \leq \frac 14 C_0^{n+1}C_1^{n} (-1)^n (n+1)! {1/2 \choose n+1} |\aa-\bb|^\gamma
\label{eq:SQG:h:15}.
\end{align}
The bounds \eqref{eq:SQG:h:5}, \eqref{eq:SQG:h:12}, and \eqref{eq:SQG:h:15} combined show that 
\begin{align*}
[ \nabla \XX(\cdot,t) ]_{C^\gamma} \leq C_0^{n+1}C_1^{n} (-1)^n (n+1)! {1/2 \choose n+1}
\end{align*}
for $0\leq t \leq T$, which concludes the proof of the H\"older estimate for $\nabla \XX$.

\subsection{Proof of the Lemma~\ref{lem:combinatorial:new}}
\label{sec:proof:combinatorial:lemma}
\begin{proof}[Proof of  identity \eqref{eq:combinatorial:new}]
In order to prove Lemma~4.2, we need to compute
\begin{align}
S_n &= \sum_{r=0}^n \sum_{m=0}^r 2 (m+1) (-1)^m \comb{m+1} (-1)^{r-m-1} \comb{r-m} (-1)^{n-r-1} \comb{n-r}\notag\\
&= \sum_{r=0}^n \sum_{m=0}^r a_m b_{r-m} b_{n-r}
\label{eq:S}
\end{align}
where $n\geq 1$, and we have defined the coefficients 
\begin{align}
a_m =  2 (m+1) (-1)^m \comb{m+1}, \qquad b_m =   (-1)^{m-1} \comb{m} 
\end{align}
for all 
$m\geq 0$. Note that both $a_m$ and $b_m$ are non-negative, and thus it is clear that $S_n\geq0$ for all $n\geq 1$. 

We now find the generating function for the coefficients $a_m$ and $b_m$.
We recall the following generalization of Newton's Binomial formula: for $\alpha \in \RR$ and $-1<t<1$, we have
\begin{align}
(1-t)^{\alpha} = 1 + \sum_{j=1}^{\infty} {\alpha \choose j} (-t)^j.
\end{align}
In particular, we have that 
\begin{align}
(1-t)^{1/2} &=  1 - \sum_{j=1}^\infty (-1)^{j-1} \comb{j} t^j = 2 - \sum_{j=0}^\infty (-1)^{j-1} \comb{j} t^j = 2 - \sum_{j=0}^\infty b_j t^j 
\label{eq:binomial:1}
\end{align}
Formally differentiating the identity \eqref{eq:binomial:1} we arrive at
\begin{align}
 \frac 12 (1-t)^{-1/2} = \sum_{j=1}^\infty j (-1)^{j-1} \comb{j} t^{j-1} = \sum_{n=0}^\infty  (n+1) (-1)^{n} \comb{n+1} t^n
\end{align}
and therefore
\begin{align}
(1-t)^{-1/2} = \sum_{j=0}^\infty a_j t^j. 
\label{eq:binomial:2}
\end{align}
Multiplying the power series formally, we now have that 
\begin{align}
 \sum_{n\geq 0} t^n \left( \sum_{r=0}^{n} \sum_{m=0}^r a_m b_{r-m} b_{n-r} \right) 
 &= \left(\sum_{j\geq 0} a_j t^j \right) \left( \sum_{j\geq 0} b_j t^j \right)^2 \notag\\
 &= (1-t)^{-1/2} \left(2- (1-t)^{1/2}\right)^2  \notag\\
 &= 4 (1-t)^{-1/2} - 2 - \left( 2 - (1-t)^{1/2} \right) \notag\\
 &= -2 + \sum_{n\geq 0} t^n \left( 4 a_n - b_n \right).
\end{align}
Equating powers of $t^n$, we thus obtain from the above that
\begin{align}
S_n = 4 a_n - b_n 
&= 8 (n+1) (-1)^n \comb{n+1} - (-1)^{n-1} \comb{n}\notag\\
&= \left(8  - \frac{2}{2n-1} \right) (n+1) (-1)^{n} \comb{n+1}  = \frac{16 n - 10}{2n-1}  (n+1) (-1)^{n} \comb{n+1} 
\label{eq:Sn:exact}
\end{align}
for all $n\geq 1$.

As a consequence, we obtain that 
\begin{align}
S_n = (n+1) (-1)^n \comb{n+1} \frac{16 n - 10}{2n-1} \leq 8 (n+1) (-1)^n \comb{n+1}
\end{align}
which completes the proof.
\end{proof}

\subsection{Proof of  estimate \texorpdfstring{\eqref{eq:K:bound}}{for the kernel}}
\label{app:kernel:bound}

The claim is that exists a universal constant $C_K \geq 1$ such that
\begin{align}
| \partial^{\alp} \KK_{in}(\yy)| \leq  \frac{C_K^{|\alp|} |\alp|!}{|\yy|^{|\alp|+2}} e^{-|\yy|^2/2}
\quad \mbox{and} \quad
| \partial^{\alp} \KK_{out}(\yy)| \leq \frac{C_K^{|\alp|} |\alp|!}{|\yy|^{|\alp|}}
\label{eq:K:bound:app}
\end{align}
holds for any multi-index $\alp$ and any $\yy\neq 0$. We shall give here the proof of the inner kernel $\KK_{in}$, since the proof for the outer kernel $\KK_{out}$ follows similarly, in view of the fact that $(1- e^{-|\yy|^2}) |\yy|^{-2} = {\mathcal O}(1)$ as $|\yy| \to 0$.

From the Leibniz rule we have
\[
\partial^{\alp} \left(\frac{\yy^\perp}{|\yy|^3} e^{-|\yy|^2}\right) = \sum_{\bet+\gam =\alp}
{\alp \choose \bet} \partial^{\bet} \left(\frac{\yy^\perp}{|\yy|^3}\right)
\partial^{\gam} (e^{-|\yy|^2})
\]
It is easy to check that the number of terms in $\partial^{\bet} \left(\frac{\yy^\perp}{|\yy|^3}\right)$
is at most $2^{|\bet|}$, and that the coefficient of each one of these terms is bounded from above by $(2|\bet|+1)!!$. Therefore, we obtain
\[
 \left|\partial^{\bet} \left(\frac{\yy^\perp}{|\yy|^3}\right)\right| \le 2^{|\bet|}\,(2|\bet|+1)!! \frac{1}{|\yy|^{|\beta|+2}}.
\]
The total number of terms in $\partial^{\gam} (e^{-|\yy|^2})$ is at most $2^{|\gam|-1}$ and the coefficient
of each term is bounded by $2^{|\gam|}$. Therefore,
\[
|\partial^{\gam} (e^{-|\yy|^2})| \le 2^{2|\gam|-1}\, e^{-|\yy|^2} \max\{1,|\yy|^{|\gam|}\}
\]
Therefore, it follows that
\begin{align*}
\left|\partial^{\alp} \left(\frac{\yy^\perp}{|\yy|^3} e^{-|\yy|^2}\right)\right| 
\leq& \sum_{\bet+\gam =\alp}
{\alp \choose \bet} 2^{|\bet|}\,(2|\bet|+1)!! \frac{1}{|\yy|^{|\beta|+2}}\,2^{2|\gam|-1}\, e^{-|\yy|^2} \max\{1,|\yy|^{|\gam|}\}\\
\leq& \frac{e^{-|\yy|^2/2} }{|\yy|^{|\alp|+2}} \,\sum_{\bet+\gam =\alp}
{\alp \choose \bet} 2^{|\bet|}\,(2|\bet|+1)!! \,2^{2|\gam|-1}\,
e^{-|\yy|^2/2}\,|\yy|^{|\gam|}\,\max\{1,|\yy|^{|\gam|}\}.
\end{align*}
Now for any $\yy\neq 0$, we have the bound
\[
e^{-|\yy|^2/2}\,|\yy|^{|\gam|}\,\max\{1,|\yy|^{|\gam|}\} \le (2|\gam|/e)^{|\gam|}.
\]
and using Stirling's formula
\[
n! \approx \sqrt{2\pi n}\, (n/e)^n, \qquad \sqrt{2\pi n}\, (n/e)^n \leq n!
\]
we arrive at
\[
e^{-|\yy|^2/2}\,|\yy|^{|\gam|}\,\max\{1,|\yy|^{|\gam|}\} \le \frac{2^{|\gam|}}{\sqrt{2\pi|\gam|}}\, 
|\gam|!
\]
Therefore, 
\begin{align*}
\left|\partial^{\alp} \left(\frac{\yy^\perp}{|\yy|^3} e^{-|\yy|^2}\right)\right| 
\leq& \frac{1}{|\yy|^{|\alp|+2}} e^{-|\yy|^2/2} \,\sum_{\bet+\gam =\alp}
{\alp \choose \bet} 2^{|\bet|}\,(2|\bet|+1)!! \,2^{2|\gam|-1}\, \frac{2^{|\gam|}}{\sqrt{2\pi|\gam|}}\,
|\gam|!\\
\le & \frac{2^{4|\alp|}\, |\alp|!}{|\yy|^{|\alp|+2}}\, e^{-|\yy|^2/2}\,
\sum_{\bet+\gam =\alp} \frac{\alp!}{\bet! \gam!}\, \frac{|\bet|!\, |\gam|!}{|\alp|!}
\end{align*}
where we have used 
\[
2^{|\bet|}\,(2|\bet|+1)!! \,2^{2|\gam|-1} \le 2^{2|\alp|}\,(|\bet|+1)! \le 2^{2|\alp|}\,(|\alp|+1)!
\le 2^{3|\alp|}\,|\alp|!
\]
Since $ |\bet|!\, |\gam|! \leq |\alp|! $, the rough estimate 
\[
\sum_{\bet+\gam =\alp} \frac{\alp!}{\bet! \gam!}\, \frac{|\bet|!\, |\gam|!}{|\alp|!} \le \sum_{\bet+\gam =\alp} \frac{\alp!}{\bet! \gam!} = 2^{|\alp|}
\]
holds.
In summary, we have shown that,
\begin{align*}
\left|\partial^{\alp} \left(\frac{\yy^\perp}{|\yy|^3} e^{-|\yy|^2}\right)\right| \le \frac{2^{5|\alp|}\, |\alp|!}{|\yy|^{|\alp|+2}}\, e^{-|\yy|^2/2}.
\end{align*}
The constant $C_K$ in \eqref{eq:K:bound} is thus less than $2^5$.

\appendix

\section{Derivation of Lagrangian formulae}
\label{sec:Lagrangian:derivation}

In this Appendix we provide the derivation of the self-contained formulae for $d\XX/dt$ and $d\nabla \XX/dt$ stated in Section~\ref{sec:examples}.
Let $\AA$ denote back-to-labels map, which is the inverse particle trajectory map, i.e.
\begin{align*}
\AA(\XX(\aa,t),t) =\aa.
\end{align*}
We will frequently use that 
\begin{align*}
(\nabla_x \AA) (\XX(\aa,t),t) (\nabla_a \XX)(\aa,t) =\II  
\end{align*}
or equivalently
\begin{align}
(\nabla_x \AA)(\XX(\aa,t),t) =((\nabla_a \XX)(\aa,t))^{-1} = (\nabla_a \XX)^\perp(\aa,t).
\label{eq:gradA:gradX}
\end{align}
Coordinate-wise the above identity is equivalent to
\begin{align*}
&\frac{\partial A_1}{\partial {x_1}}(\XX(\aa,t),t) = \frac{\partial X_2}{\partial a_2}(\aa,t),\qquad
\frac{\partial A_2}{\partial {x_1}}(\XX(\aa,t),t) = -\frac{\partial X_2}{\partial a_1}(\aa,t),\\
&\frac{\partial A_1}{\partial {x_2}}(\XX(\aa,t),t) = -\frac{\partial X_1}{\partial a_2} (\aa,t),\qquad
\frac{\partial A_2}{\partial {x_2}}(\XX(\aa,t),t) = \frac{\partial X_1}{ \partial a_1} (\aa,t).
\end{align*}
The upshot of the above formulae is that if we define
\begin{align*} 
\theta_0(\AA(\xx,t)) = \theta(\xx,t)
\end{align*}
then we have
\begin{align} 
\partial_{x_k} \theta(\xx,t) = \frac{\partial \theta_0}{\partial {a_j}}(\AA(\xx,t)) \frac{\partial A_j}{\partial {x_k}}(\xx,t) = \frac{\partial \theta_0}{\partial a_j}(\aa) \frac{\partial X_k^\perp}{\partial a_j^\perp}(\aa,t) 
\label{eq:gradx:gradX}
\end{align}
where in the last equality we have used \eqref{eq:gradA:gradX}.

\subsection{2D SQG}
\label{sec:app:SQG}
The constitutive law of SQG yields 
\[
\uu(\xx) = {\mathcal R}^\perp \theta(\xx) = \int \frac{(\xx-\yy)^\perp}{2\pi |\xx-\yy|^3} \theta(\yy) d\yy = \int \KK(\xx-\yy) \theta(\yy) d\yy
\] 
and the evolution gives
\[
\theta(\XX(\bb,t),t) = \theta_0(\bb)
\]
Combining the above we arrive at
\begin{align*} 
\frac{d\XX}{dt} (\aa,t) 
=   \int \KK(\XX(\aa,t) - \yy)  \theta(\yy,t) d\yy 
=   \int \KK(\XX(\aa,t) - \XX(\bb,t))  \theta_0(\bb) d\bb
\end{align*}
since by incompressibility the determinant of the Jacobian is equal to $1$. To derive the formula for $d (\nabla \XX)/dt$, 
we switch back to Eulerian coordinates where
\[
\partial_{x_k} u_i(\xx) = \int \KK(\xx-\yy) \partial_{y_k} \theta(\yy,t) d\yy
\]
and then appeal to \eqref{eq:gradx:gradX} in order to obtain
\begin{align*} 
\frac{d}{dt} \frac{\partial X_i}{\partial a_j} (\aa,t) = \frac{\partial X_k}{\partial a_j} \int K_i(\XX(\aa,t)-\XX(\bb,t)) \frac{\partial \theta_0}{\partial b_j}(\bb)  \frac{\partial A_j}{\partial y_k} (\XX(\bb,t),t) d\bb.
\end{align*}
Using \eqref{eq:gradA:gradX} we arrive at
\begin{align*} 
\frac{d(\nabla_a \XX)}{dt} (\aa,t) = \nabla_a \XX(\aa,t) \int \KK(\XX(\aa,t)-\XX(\bb,t)) (\nabla_b^\perp \XX^\perp) (\bb,t) {(\nabla_b \theta_0)}(\bb) d\bb.
\end{align*}
which proves \eqref{eq:SQG:nabla:X:def}.

\subsection{2D IPM}
\label{sec:app:IPM}
In Eulerian coordinates the scalar vorticity $\omega$ satisfies
\begin{align*}
\omega = \nabla^\perp \cdot \uu = -\partial_{x_1}\theta.
\end{align*}
Therefore, along particle trajectories we have
\begin{align*}
\omega(\XX(\aa,t),t) = -(\partial_{x_1}\theta)(\XX(\aa,t),t) =-\left\{\theta_0(\aa), X_2(\aa,t)\right\}.
\end{align*}
Therefore, since the kernel of the two dimensional Biot-Savart law in Eulerian coordinates is given by
\[
\uu(\xx) = \frac{1}{2\pi} \int \frac{(\xx-\yy)^\perp}{|\xx-\yy|^2} \omega(\yy) d\yy,
\]
upon letting $\yy = \XX(\bb,t)$ we obtain
\begin{align*}
\frac{d \XX}{dt}(\aa,t) 
&=  \frac1{2\pi} \int \frac{(\XX(\aa,t)-\XX(\bb,t))^\perp}{|(\XX(\aa,t)-\XX(\bb,t)|^2}\,\omega(\XX(\bb,t),t)\,d\bb\\
&= -\frac1{2\pi} \int \frac{(\XX(\aa,t)-\XX(\bb,t))^\perp}{|(\XX(\aa,t)-\XX(\bb,t)|^2}\,\left\{\theta_0(\bb), X_2(\bb,t) \right\}\,d\bb.
\end{align*}
To derive the formula for $\partial_t \nabla X$, we  differentiate the kernel and arrive at
\begin{align}
\frac{d(\nabla_a \XX)}{dt} (\aa,t) 
&=- \nabla_a \XX(\aa,t)  \int \KK(\XX(\aa,t)-\XX(\bb,t))\,   \left\{\theta_0(\bb), X_2(\bb,t)\right\}  \,d\bb  \notag \\
&\qquad + \frac12\left\{\theta_0(\aa), X_2(\aa,t)\right\}\,\left[\begin{array}{cc} 0  & -1\\ 1 & 0\end{array}  \right]\nabla_a \XX(\aa,t)
\label{eq:IPM:nablaX}
\end{align}
where $\KK$ is the same as in \eqref{kernelK}, namely
\begin{align}
\KK(\yy)= \KK(y_1, y_2) = \frac1{2\pi |\yy|^4} \left[\begin{array}{cc} 2y_1 \,y_2 & y_2^2-y_1^2\\ y_2^2-y_1^2 & -2y_1 \,y_2\end{array}  \right].
\label{eq:kernelK}
\end{align}

\subsection{3D Euler}
\label{sec:app:3DE}

From the Biot-Savart in three dimensions 
\[
\uu(\xx,t) = \fr{1}{4\pi}\int_{\RR^3}\frac{\xx-\yy}{|\xx-\yy|^3}\times \om(\yy,t)d\yy.
\]
composition with the Lagrangian path $\yy = \XX(\bb,t)$, and the Cauchy formula
\[
\om (\XX(\aa,t), t) = \na\XX(\aa,t)\om_0(\aa)
\]
we arrive at a self-contained formula for the evolution of $\XX(\aa,t)$
\[
\frac{d\XX}{dt}(\aa,t) = \frac{1}{4\pi} \int \frac{\XX(\aa,t)-\XX(\bb,t)}{|\XX(\aa,t)-\XX(\bb,t)|^3} \times ( \nabla_b \XX(\bb,t) \om_0(\bb))d\bb.
\]
The evolution equation for $\nabla \XX$ is obtained by first switching to Eulerian coordinates, which allows us to compute $\nabla_x \uu$ from $\om$ via Calder\'on-Zygmund singular integrals. 
For this purpose one considers the rate of strain matrix
\[
S_{ij} = \fr{1}{2}\left (\pa_i u_j + \pa_j u_i\right)
\]
and uses the Biot-Savart law to compute
\begin{align*}
S_{ij} &= \frac{3}{8\pi} \int_{\RR^3}  \frac{\left( (\xx-\yy)\times \om(\yy)\right)_i (\xx-\yy)_j + \left( (\xx-\yy)\times \om(\yy)\right)_j (\xx-\yy)_i}{|\xx-\yy|^5} d\yy \notag \\
&=: \int ( \KK(\xx-\yy) \om(\yy))_{ij} dy
\end{align*}
where we have defined
\begin{align}
( \KK(\xx) \yy)_{ij} = \frac{3}{8\pi} \frac{\left( \xx \times \yy \right)_i x_j + \left( \xx \times \yy\right)_j x_i}{|\xx|^5}.
\la{so}
\end{align}
Of course, the full gradient is then obtain using
\[
(\na \uu)\boldsymbol{v} = \boldsymbol{S} \boldsymbol{v} + \fr{1}{2}\om  \times \boldsymbol{v}.
\] 
To obtain the evolution of $\nabla \XX$ we then compute
\begin{align*}
\frac{d}{dt} \frac{\partial X_i}{\partial a_j}(\aa,t) 
&= \frac{\partial u_i}{\partial x_k}(\XX(\aa,t),t) \frac{\partial X_k}{\partial a_j}(\aa,t) \\
&= S_{ik}(\XX(\aa,t),t) \frac{\partial X_k}{\partial a_j}(\aa,t) + \frac 12 ( \om(\XX(\aa,t),t) \times (\nabla_{a_j} \XX)(\aa,t) )_{i}\\
&= \int \left[ \KK(\XX(\aa,t)-\XX(\bb,t)) \left( \nabla_b \XX(\bb,t) \om_0(\bb) \right) \right]_{ik} d\yy \frac{\partial X_k}{\partial a_j}(\aa,t) \notag\\ 
&\qquad + \frac 12 \left( (\nabla_a\XX(\aa,t)\om_0(\aa)) \times (\nabla_{a_j} \XX)(\aa,t) \right)_{i}
\end{align*}
where we have used the notation in \eqref{so} for the $ik$-component of $\KK(\cdot) (\nabla_a \XX \om_0)$.

\subsection{2D Euler}
\label{sec:app:2DE}
From the Lagrangian conservation 
\[
\omega(\XX(\aa,t),t)= \omega_0(\aa)
\]
and the Eulerian two dimensional Biot-Savart law~\cite{MajdaBertozzi02}
we directly arrive at
\begin{align*} 
\frac{d\XX}{dt}(\aa,t) = \frac{1}{2\pi} \int \frac{(\XX(\aa,t)-\XX(\bb,t))^\perp}{|\XX(\aa,t)-\XX(\bb,t)|^2} \omega_0(\bb) d\bb.
\end{align*}
Estimates for the time derivative of $\nabla_a \XX$ are obtained from the above by differentiating the kernel, similarly to \eqref{eq:IPM:nablaX}. We obtain
\begin{align*} 
\frac{d(\nabla_a \XX)}{dt}(\aa,t)
&= \nabla_a \XX(\aa,t)  \int \KK(\XX(\aa,t)-\XX(\bb,t))\, \omega_0(\bb) \,d\bb  + \frac12 \omega_0(\aa) \,\left[\begin{array}{cc} 0  & -1\\ 1 & 0\end{array}  \right]\nabla_a \XX(\aa,t)
\end{align*}
where the kernel $\KK$ is given in \eqref{eq:kernelK}.

\subsection{2D Boussinesq}
\label{sec:app:Boussinesq}

Along the particle trajectory $\xx=\XX(\aa,t)$, the vorticity obeys
\begin{align*}
\partial_t \omega(\XX(\aa,t), t) = (\partial_{x_1} \theta)(\XX(\aa,t), t).
\end{align*}
Integrating in time yields
\begin{align*}
\omega(\XX(\aa,t), t) = \omega_0(\aa) + \int_0^t (\partial_{x_1} \theta)(\XX(\aa,\tau), \tau)\,d\tau.
\end{align*}
Next, we rewrite $(\partial_{x_1} \theta)(\XX(\aa,\tau), \tau)$ in terms of the Lagrangian coordinates.
The equation for $\theta$ yields
\begin{align*}
\theta(\xx,t) =\theta_0(\AA(\xx,t)).
\end{align*}
Therefore, we have
\begin{align*}
(\partial_{x_1} \theta)(\xx,t) = \frac{\partial \theta_0}{\partial {a_1}}(\AA(\xx,t)) \frac{\partial A_1}{\partial {x_1}}(\xx,t)
+ \frac{\partial \theta_0}{\partial {a_2}}(\AA(\xx,t)) \frac{\partial A_2}{\partial {x_1}}(\xx,t),
\end{align*}
and letting $\xx=\XX(\aa,t)$ yields
\begin{align*}
(\partial_{x_1} \theta) (\XX(\aa,t),t) = \frac{\partial \theta_0}{\partial {a_1}}(\aa) \frac{\partial A_1}{\partial {x_1}}(\XX(\aa,t),t)
+ \frac{\partial \theta_0}{\partial {a_2}}(\aa) \frac{\partial A_2}{\partial {x_1}}(\XX(\aa,t),t).
\end{align*}
Upon using \eqref{eq:gradA:gradX} we arrive at
\begin{align*}
(\partial_{x_1} \theta) (\XX(\aa,t),t) = \partial_{a_1} \theta_0(\aa) \partial_{a_2} X_2(\aa,t)
- \partial_{a_2} \theta_0(\aa) \partial_{a_1} X_2(\aa,t) 
=\left\{\theta_0(\aa), X_2(\aa,t) \right\},
\end{align*}
and therefore 
\begin{align*}
\omega(\XX(\aa,t), t) = \omega_0(\aa) + \int_0^t \left\{\theta_0(\aa), X_2(\aa,\tau)\right\}\,d\tau.
\end{align*}
To obtain and equation just in terms of $\XX$, we recall
\begin{align*}
\frac{d \XX}{dt}(\aa,t) =\uu(\XX(\aa,t),t) = \frac1{2\pi} \int \frac{(\XX(\aa,t)-\XX(\bb,t))^\perp}{|(\XX(\aa,t)-\XX(\bb,t)|^2}\, \omega(\XX(\bb,t), t)\,d\bb
\end{align*}
Therefore,
\begin{align*}
\frac{d \XX}{dt}(\aa,t) 
&= \frac1{2\pi} \int \frac{(\XX(\aa,t)-\XX(\bb,t))^\perp}{|(\XX(\aa,t)-\XX(\bb,t)|^2}\, \omega_0(\bb)\, d\bb \notag \\
&\quad + \frac1{2\pi} \int \frac{(\XX(\aa,t)-\XX(\bb,t))^\perp}{|(\XX(\aa,t)-\XX(\bb,t)|^2}
\left( \int_0^t \left\{\theta_0(\bb), X_2(\bb,\tau)\right\} d\tau \right)d\bb.
\end{align*}
To derive the formula for $\partial_t \nabla \XX$, we  differentiate the kernel and obtain
\begin{align*}
\frac{d(\nabla_a \XX)}{dt} (\aa,t) 
&= \left(  \int \KK(\XX(\aa,t)-\XX(\bb,t))\,\omega_0(\bb) \, d\bb\right) \nabla_a \XX(\aa,t) \notag \\
&\quad +  \left(  \int \KK(\XX(\aa,t)-\XX(\bb,t)) \int_0^t  \left\{\theta_0(\bb), X_2(\bb,\tau)\right\}   \,d\tau\,d\bb\right)\nabla_a \XX(\aa,t) \notag \\
&\quad + \frac12\left(\omega_0(\aa) + \int_0^t \left\{\theta_0(\aa), X_2(\aa,\tau)\right\}\,d\tau\right)\,\left[\begin{array}{cc} 0  & -1\\ 1 & 0\end{array}  \right]\nabla_a \XX(\aa,t),
\end{align*}
where $\KK$ is given in  \eqref{eq:kernelK} above.


\section{The composition of analytic functions: the one dimensional case}
\label{sec:1D}
The contents of this section is adapted from~\cite[Theorem 1.3.2]{KrantzParks02}, and is presented here for the sake of completeness. This serves as the motivation for the combinatorial machinery given in Section~\ref{sec:composition} above.

\begin{proposition}\label{prop:1D}
If $g\colon \RR \to \RR$ is bounded $h\colon \RR \to \RR$ is real analytic, and $g$ obeys the ODE
\begin{align}
g'(x) = h(g(x)),
\label{eq:ODE:1}
\end{align} 
then $g$ is in fact real analytic.
\end{proposition}

\begin{lemma}[\bf One-dimensional Fa\`a di Bruno formula]
\label{lem:FaaDiBruno}
Let $I \subset \RR$ be an open interval, $g\in C^\infty(I)$, and $h\in C^\infty(J)$, where $J= f(I)$. Let $f = h\circ g$. Then for all $n\geq 1$ we have
\begin{align*}
f^{(n)}(x) = {\sum_{k=1}^n} {h^{(k)}(g(x)) } \sum_{\kk \in P(n;k)} \frac{n!}{\kk !} \prod_{j=1}^{n} \left( \frac{g^{(j)}(x)}{j!} \right)^{k_j}
\end{align*}
where $\kk = (k_1,\ldots,k_n)$ is a multi-index,
\begin{align*}
P(n,k) = \left\{ \kk=(k_1,\ldots,k_n) \colon \sum_{j=1}^{n} j k_j = n, \sum_{j=1}^{n} k_j = k \right\}
\end{align*}
and we use the notation
\begin{align*}
\kk! = k_1! \ldots k_n!
\end{align*}
\end{lemma}

A consequence of the Fa\`a di Bruno formula is the following identity, as given in~\cite[Lemma 1.5.2]{KrantzParks02}.
\begin{lemma}[\bf One-dimensional magic identity]
\label{lem:Magic}
For each integer $n\geq 1$ we have
\begin{align*}
{\sum_{k=1}^n} \sum_{\kk \in P(n,k)} \frac{(-1)^k k!}{\kk!} \prod_{j=1}^n {1/2 \choose j} ^{k_j} = 2 (n+1) {1/2 \choose n+1}.
\end{align*}
\end{lemma}

\begin{proof}[Proof of Proposition~\ref{prop:1D}]
The assumption that $h$ is real analytic translates into the fact that there {exists} $C,R>0$ such that
\begin{align}
| h^{(k)}(y) | \leq C \frac{k!}{R^k}
\label{eq:h}
\end{align}
for all $k \geq 0$, and all $y$ close to some $y_0$.

We make the following inductive assumption on the function $g$: that for all {$1\leq j \leq n$ it holds that}
\begin{align}
|g^{(j)}(x)| \leq \frac{1}{R} j! (-1)^{j-1} {1/2 \choose j} \left(\frac{2C}{R} \right)^j
\label{eq:g:assumption}
\end{align}
at all points $x$ sufficiently close to some $x_0$.

Let $n\geq 0$. We apply $n$ derivatives to the equation \eqref{eq:ODE} and use Lemma~\ref{lem:FaaDiBruno} to obtain
\begin{align*}
g^{(n+1)}(x) = \sum_{\kk \in P(n;k)} \frac{n!}{\kk !} h^{(k)}(g(x)) \prod_{j=1}^{n} \left( \frac{g^{(j)}(x)}{j!} \right)^{k_j}
\end{align*}
We appeal to \eqref{eq:h} and the inductive assumption \eqref{eq:g:assumption} to estimate
\begin{align*}
|g^{(n+1)}|
&\leq C \sum_{\kk \in P(n;k)} \frac{n!}{\kk !} \frac{k!}{R^k} \prod_{j=1}^{n} \left( (-1)^{j-1} {1/2 \choose j} \frac{(2C)^j}{R^{j-1}} \right)^{k_j}.
\end{align*}
Using that $\sum_j k_j = k$ and ${\sum_j} j k_j = n$ we obtain that
\begin{align*}
|g^{(n+1)}| &\leq C n! (-1)^n \frac{(2C)^n}{R^n} \sum_{\kk \in P(n;k)} \frac{(-1)^k k!}{\kk !}  \prod_{j=1}^{n} {1/2 \choose j}^{k_j}.
\end{align*}
Using the identity given in Lemma~\ref{lem:Magic} we thus obtain
\begin{align*}
|g^{(n+1)}| &\leq C n! (-1)^n \frac{(2C)^n}{R^n}  2 (n+1) {1/2 \choose n+1}\notag\\
&= (n+1)! (-1)^n \frac{(2C)^{n+1}}{R^n} {1/2 \choose n+1}
\end{align*}
which is exactly \eqref{eq:g:assumption} at level $n+1$. This completes the proof since in view of \eqref{eq:factorial:bound}, the bound \eqref{eq:g:assumption} gives
\begin{align*}
|g^{(j)}(x)| \leq \frac{C}{R} \frac{j!}{(R/C)^j}
\end{align*}
which shows that $g$ is real analytic with radius of convergence $R/C$.
\end{proof}

\section*{Acknowledgments} 
The work of PC was supported in part by the NSF grants DMS-1209394 and DMS-1265132, 
VV was supported in part by the NSF grant DMS-1211828, 
while the work of JW was supported in part by the NSF grant DMS-1209153.


\end{document}